\documentclass[12pt]{amsart}

\usepackage{amsfonts,amsmath,amsthm,amssymb,latexsym,amsthm,newlfont,enumerate,color}

\newif\ifslide

\theoremstyle{plain}

\ifdefined\spacer
\newtheorem{theorem}{Theorem}
 \else
\newtheorem{theorem}{Theorem}[section]
\fi

\newtheorem{theorema}{Theorem}

\newtheorem{corollary}[theorem]{Corollary}
\newtheorem{lemma}[theorem]{Lemma}
\newtheorem{claim}[theorem]{Claim}

\newtheorem{definition-lemma}[theorem]{Definition-Lemma}

\newtheorem{red-question}[theorem]{\textcolor{red}{Question}}

\theoremstyle{definition}
\newtheorem{definition}[theorem]{Definition}
\newtheorem{remark}[theorem]{Remark}
\newtheorem{setup}[theorem]{Setup}

\def\ideal#1.{I_{#1}}
\def\ring#1.{\mathcal {O}_{#1}}

\def\ddiv{\operatorname {div}}
\def\base{\operatorname {Bs}}

\def\fring#1.{\hat{\mathcal {O}}_{#1}}
\def\proj#1.{\mathbb {P}(#1)}
\def\pr #1.{\mathbb {P}^{#1}}
\def\dpr #1.{\hat{\mathbb {P}}^{#1}}
\def\af #1.{\mathbb A^{#1}}
\def\Hz #1.{\mathbb F_{#1}}
\def\Hbz #1.{\overline{\mathbb F}_{#1}}
\def\fb#1.{\underset #1 {\times}}
\def\rest#1.{\underset {\ \ring #1.} \to \otimes}
\def\au#1.{\operatorname {Aut}\,(#1)}
\def\deg#1.{\operatorname {deg } (#1)}

\def\pic#1.{\operatorname {Pic}\,(#1)}
\def\pico#1.{\operatorname{Pic}^0(#1)}
\def\picg#1.{\operatorname {Pic}^G(#1)}
\def\ner#1.{NS (#1)}
\def\rdown#1.{\llcorner#1\lrcorner}
\def\rfdown#1.{\lfloor{#1}\rfloor}
\def\rup#1.{\ulcorner{#1}\urcorner}
\def\rcup#1.{\lceil{#1}\rceil}

\def\n1#1.{\operatorname {N_1}(#1)}  
\def\cn1#1.{\overline{\operatorname {N^1}(#1)}} 
\def\cone#1.{\operatorname {NE}(#1)}     
\def\ccone#1.{\overline{\operatorname {NE}}(#1)}
\def\none#1.{\operatorname {NF}(#1)}
\def\cnone#1.{\overline{\operatorname {NF}}(#1)}
\def\mone#1.{\operatorname {NM}(#1)} 
\def\cmone#1.{\overline{\operatorname {NM}}(#1)}

\def\coef#1.{\frac{(#1-1)}{#1}}
\def\vit#1.{D_{\langle #1 \rangle}}
\def\mm#1.{\overline {M}_{0,#1}}
\def\H1#1.{H^1(#1,{\ring #1.})}
\def\ac#1.{\overline {\mathbb F}_{#1}}

\def\adj#1.{\frac {#1-1}{#1}}
\def\spn#1.{\overline{#1}}
\def\pek#1.#2.{\Cal P^{#1}(#2)}
\def\plk#1.#2.{\Cal P^{\leq #1}(#2)}
\def\ev#1.{\operatorname{ev_{#1}}}
\def\ilist#1.{{#1}_1,{#1}_2,\dots}
\def\bminv#1.{(\nu_1,s_1;\nu_2,s_2;\dots ;\nu_{#1},s_{#1};\nu_{r+1})}
\def\zinv#1.{(\nu_1,s_1;\nu_2,s_2;\dots ;\nu_{#1},s_{#1};0)}
\def\iinv#1.{(\nu_1,s_1;\nu_2,s_2;\dots ;\nu_{#1},s_{#1};\infty)}

\def\scr #1.{\mathcal #1}

\def\llist#1.#2.{{#1}_1,{#1}_2,\dots,{#1}_{#2}}
\def\ulist#1.#2.{{#1}^1,{#1}^2,\dots,{#1}^{#2}}
\def\lomitlist#1.#2.{{#1}_1,{#1}_2,\dots,\hat {{#1}_i}, \dots, {#1}_{#2}}
\def\lomitlistz#1.#2.{{#1}_0,{#1}_1,\dots,\hat {{#1}_i}, \dots, {#1}_{#2}}
\def\loc#1.#2.{\Cal O_{#1,#2}}
\def\fderiv#1.#2.{\frac {\partial #1}{\partial #2}}
\def\deriv#1.#2.{\frac {d #1}{d #2}}

\def\map#1.#2.{#1 \longrightarrow #2}
\def\rmap#1.#2.{#1 \dasharrow #2}
\def\emb#1.#2.{#1 \hookrightarrow #2}
\def\non#1.#2.{\text {Spec }#1[\epsilon]/(\epsilon)^{#2}}
\def\Hi#1.#2.{\text {Hilb}^{#1}(#2)}
\def\sym#1.#2.{\operatorname {Sym}^{#1}(#2)}
\def\Hb#1.#2.{\text {Hilb}_{#1}(#2)}
\def\Hm#1.#2.{\Hom_{#1}(#2)}
\def\prd#1.#2.{{#1}_1\cdot {#1}_2\cdots {#1}_{#2}}
\def\Bl #1.#2.{\operatorname {Bl}_{#1}#2}
\def\pl #1.#2.{#1^{\otimes #2}}
\def\mgn#1.#2.{\overline {M}_{#1,#2}}
\def\ialist#1.#2.{{#1}_1 #2 {#1}_2, #2\dots}
\def\pair#1.#2.{\langle #1, #2\rangle}
\def\vandermonde#1.#2.{\left|
\begin{matrix}
1 & 1 & 1 & \dots & 1\\
{#1}_1 & {#1}_2 & {#1}_3 & \dots & {#1}_{#2}\\
{#1}_1^2 & {#1}_2^2 & {#1}_3^2 & \dots & {#1}_{#2}^2\\
\vdots & \vdots & \vdots & \ddots & \vdots\\
{#1}_1^{#2-1} & {#1}_2^{#2-1} & {#1}_2^{#2-1} & \dots & {#1}_{#2}^{#2-1}\\
\end{matrix}
\right|
}
\def\vandermondet#1.#2.{\left|
\begin{matrix}
1 & {#1}_1   & {#1}_1^2 & \dots & {#1}_1^{#2-1}\\
1 & {#1}_2   & {#1}_2^2 & \dots & {#1}_2^{#2-1}\\
1 & {#1}_3   & {#1}_3^2 & \dots & {#1}_3^{#2-1}\\
\vdots & \vdots & \vdots & \ddots & \vdots\\
1 & {#1}_{#2}& {#1}_{#2}^2 & \dots & {#1}_{#2}^{#2-1}\\
\end{matrix}
\right|
}
\def\gr#1.#2.{\mathbb{G}(#1,#2)}

\def\alist#1.#2.#3.{{#1}_1 #2 {#1}_2 #2\dots #2 {#1}_{#3}}
\def\zlist#1.#2.#3.{#1_0 #2 #1_1 #2\dots #2 #1_{#3}}
\def\lomitlist30#1.#2.#3.{{#1}_0,{#1}_1 #2 \dots #2\hat {{#1}_i} #2\dots #2 {#1}_{#3}}
\def\lmap#1.#2.#3.{#1 \overset{#2}{\longrightarrow} #3}
\def\mes#1.#2.#3.{#1 \longrightarrow #2 \longrightarrow #3}
\def\ses#1.#2.#3.{0\longrightarrow #1 \longrightarrow #2 \longrightarrow #3 \longrightarrow 0}
\def\les#1.#2.#3.{0\longrightarrow #1 \longrightarrow #2 \longrightarrow #3}
\def\res#1.#2.#3.{#1 \longrightarrow #2 \longrightarrow #3\longrightarrow 0}
\def\Hi#1.#2.#3.{\text {Hilb}^{#1}_{#2}(#3)}
\def\ten#1.#2.#3.{#1\underset {#2}{\otimes} #3}
\def\lomitlist30#1.#2.#3.{{#1}_0 #2 {#1}_1 #2 \dots #2 \hat {{#1}_i} #2 \dots #2 {#1}_{#3}}
\def\mderiv#1.#2.#3.{\frac {d^{#3} #1}{d #2^{#3}}}

\def\Hom{\operatorname{Hom}}

\def\Supp{\operatorname{Supp}}
\def\Bs{\operatorname{\mathbf B}}
\def\Exc{\operatorname{Exc}}

\def\dim{\operatorname{dim}}

\def\deg{\operatorname{deg}}

\def\ker{\operatorname{Ker}}

\def\im{\operatorname{Im}}

\def\Div{\operatorname{Div}}

\def\mult{\operatorname{mult}}
\def\face{\operatorname{face}}

\def\mob{\operatorname{Mob}}
\def\fix{\operatorname{Fix}}
\def\bfix{\operatorname{\mathbf{Fix}}}
\def\F{\operatorname{\mathbf{F}}}

\def\rest{\operatorname{res}}

\def\bs{\operatorname{Bs}}

\def\C{\mathbb C}

\def\e{\Cal E}

\def\e1{E_1}
\def\e2{E_2}

\def\mapdown#1{\big\downarrow\rlap{$\vcenter{\hbox{$\scriptstyle#1$}}$}}

\def\mapse#1{
{\vcenter{\hbox{$\mathop{\smash{\raise1pt\hbox{$\diagdown$}\!\lower7pt
\hbox{$\searrow$}}\vphantom{p}}\limits_{#1}\vphantom{\mapdown{}}$}}}}

\def\VR#1.{height#1pt&\omit&&\omit&&\omit&&\omit&&\omit&\cr}

\def\VRT#1.{height#1pt&\omit&&\omit&\cr}

\usepackage{hyperref}

\title[New outlook on the MMP, I]{New outlook on the Minimal Model Program, I}

\author{Paolo Cascini}
\address{Department of Mathematics\\
Imperial College London\\
180 Queen's Gate\\
London SW7 2AZ, UK}
\email{p.cascini@imperial.ac.uk}

\author{Vladimir Lazi\'c}
\address{Mathematisches Institut\\
Universit\"at Bayreuth\\
95440 Bayreuth\\
Germany}
\email{vladimir.lazic@uni-bayreuth.de}

\thanks{Part of this work was written while the second author was a
PhD student of A.~Corti, who influenced ideas
developed here immensely. Part of the paper started as a collaboration with J.~M\textsuperscript cKernan.
We would like to express our gratitude to both of them for their encouragement, support and continuous inspiration.
We thank F.~Ambro, C.~Hacon, J.~Hausen, A.-S.~Kaloghiros, A.~Lopez, K.~Matsuki, and M.~Reid for many useful comments. We are  particularly grateful to the referees who helped to improve the presentation of the paper considerably. \\
\indent The first author was partially supported by an EPSRC grant. The second author is grateful for support from
the University of Cambridge, the Max-Planck-Institut f\"ur Mathematik, and the Institut Fourier.}

\begin{document}

\begin{abstract}
We give a new and self-contained proof of the finite generation of adjoint rings with big boundaries. As a consequence,
we show that the canonical ring of a smooth projective variety is finitely generated.
\end{abstract}

\maketitle
\tableofcontents

\section{Introduction}
\label{s_introduction}

The main goal of this paper is to provide a new proof of the following theorem while avoiding techniques of the Minimal Model Program.

\begin{theorem}\label{t_finite} Let $X$ be a smooth projective variety and let
$\Delta$ be a $\mathbb{Q}$-divisor with simple normal crossings such that
$\rfdown\Delta.=0$.

Then the log canonical ring $R(X,K_X+\Delta)$ is finitely generated.
\end{theorem}

This work supersedes \cite{Laz09}, where the results of this paper were first proved without the Minimal Model Program by the second author.
Several arguments here follow closely those in \cite{Laz09} and,
based on these methods, we obtain a streamlined proof which is almost entirely self-contained.
We even prove a lifting statement for adjoint bundles without relying on asymptotic multiplier ideals,
assuming only Kawamata-Viehweg vanishing and some elementary arithmetic -- this is Theorem \ref{t_lifting}, which slightly generalises the lifting theorem from \cite{HM10}.

The results presented here were originally proved by extensive use of methods of the Minimal Model Program in \cite{BCHM10,HM10}, and an analytic proof of finite generation of the canonical ring
for varieties of general type is announced in \cite{Siu08}. By contrast, in this paper we
avoid the following tools which are commonly used in the Minimal Model Program:
Mori's bend and break, which relies on methods in positive characteristic \cite{Mori82}, the Cone and Contraction theorem \cite{KM98},
the theory of asymptotic multiplier ideals, which was necessary to prove the existence of flips in
\cite{HM10}. Moreover, contrary to classical Minimal Model Program, we do not need to
work with singular varieties.

In \cite{CL10}, Corti and the second author recently proved that the Cone and Contraction theorem, and the main result of \cite{BCHM10}, follow quickly
from one of our main results, Theorem \ref{t_cox}. Therefore, this paper and \cite{CL10} together give a completely new organisation of the Minimal Model Program.

We now briefly describe the strategy of the proof. As part of the induction, we prove the following two theorems.

\begin{theorema}\label{t_cox}
Let $X$ be a smooth projective variety of dimension $n$. Let $B_1,\dots,B_k$ be $\mathbb{Q}$-divisors on $X$ such that $\rfdown B_i.=0$ for all $i$, and such that the support of $\sum_{i=1}^k B_i$
has simple normal crossings. Let $A$ be an ample $\mathbb{Q}$-divisor on $X$, and denote $D_i=K_X+A+B_i$ for every $i$.

Then the adjoint ring
$$
R(X;D_1,\dots,D_k)=\bigoplus_{(m_1,\dots,m_k)\in\mathbb{N}^k} H^0\big(X,\ring X. \big(\rfdown\textstyle\sum m_iD_i.\big)\big)
$$
is finitely generated.
\end{theorema}

\begin{theorema}\label{t_non-vanishing}
Let $(X,\sum_{i=1}^p S_i)$ be a log smooth projective pair of dimension $n$, where
$S_1,\dots,S_p$ are distinct prime divisors. Let $V=\sum_{i=1}^p \mathbb R S_i\subseteq \Div_{\mathbb R}(X)$, let $\mathcal L(V)=\{ B=\textstyle\sum b_iS_i\in V \mid 0\le b_i\le 1\text{ for
all }i\}$,
and let $A$ be an ample $\mathbb{Q}$-divisor on $X$.

Then
$$\mathcal{E}_A(V)=\{B\in\mathcal L(V)\mid |K_X+A+B|_{\mathbb
R}\neq\emptyset\}
$$
is a rational polytope.
\end{theorema}

Note that all the results in this paper hold, with the same proofs, when varieties are projective over affine varieties.
For definitions of various terms involved in the statements of the theorems, see Section \ref{s_preliminary}.
 In the sequel, ``Theorem \ref{t_cox}$_n$'' stands for ``Theorem \ref{t_cox} in dimension $n$,'' and so forth.

In Section \ref{s_preliminary} we lay the foundation for the remainder of the paper: we discuss basic properties of
asymptotic invariants of divisors, convex geometry and Diophantine approximation, and we introduce divisorial rings graded
by monoids of higher rank and present basic consequences of finite generation of these rings.
Basic references for asymptotic invariants of divisors are \cite{Nakayama04,ELMNP}. The first systematic use of
Diophantine approximation in the Minimal Model Program was initiated by Shokurov in \cite{Shokurov03}, and our arguments
at several places in this paper are inspired by some of the techniques introduced there.

In Section \ref{s_lifting} we give a simplified proof of a version of the lifting lemma from \cite{HM10}. The proof in \cite{HM10}
is based on methods initiated in \cite{Siu98}, which also inspired a systematic use of multiplier ideals. We want to emphasise that our proof, even though ultimately
following the same path, is much simpler and uses only Kawamata-Viehweg vanishing and some elementary arithmetic.

In Section \ref{s_main-lemma} we prove that one of the sets which naturally appears in the theory is a rational polytope. Some steps in the proof are
close in spirit to Hacon's ideas in the proof of \cite[Theorem 9.16]{HK10}. The proof is an application of the lifting result from
Section \ref{s_lifting}.

In Section \ref{s_effective} we prove Theorem \ref{t_non-vanishing}$_n$, assuming Theorems \ref{t_cox}$_{n-1}$ and \ref{t_non-vanishing}$_{n-1}$.
Certain steps of the proof here are similar to \cite[Section 6]{BCHM10}, and they rely on Nakayama's techniques from \cite{Nakayama04}.
Lemma \ref{l_nonvanishing} was obtained in \cite{Pau08} by analytic
methods, without assuming Theorems \ref{t_cox}$_{n-1}$ and \ref{t_non-vanishing}$_{n-1}$. We remark here that several arguments of this section can be made somewhat shorter if one were to assume some facts about lengths of extremal rays, similarly as in \cite{BCHM10}; however, we are deliberately making the proofs a bit longer by proving everything ``from scratch", especially since one of the aims of this paper is to provide the basis for simpler proofs of the foundational results of the Minimal Model Program \cite{CL10}.

Finally, in Section \ref{s_finite}, we prove Theorem \ref{t_cox}$_n$, assuming Theorems \ref{t_cox}$_{n-1}$ and \ref{t_non-vanishing}$_n$, therefore
completing the induction step. This part of the proof is close in spirit to that of the finite generation of the
restricted ring when the grading is by the non-negative integers, see \cite[Lemma 2.3.6]{Corti05}.

The papers \cite{Corti11} and \cite{CL11} give an introduction to some of the ideas presented in this work.

\section{Preliminary results}
\label{s_preliminary}

\subsection{Notation and conventions}

In this paper all algebraic varieties are defined over
$\mathbb C$. We denote by $\mathbb R_+$ and $\mathbb Q_+$ the sets of non-negative
real and rational numbers. For any $x,y\in \mathbb R^N$, we denote by $[x,y]$ and $(x,y)$ the closed and open segments joining $x$ and $y$. Given subsets $A,B\subseteq \mathbb R^N$, the Minkowski sum of $A$ and $B$ is  
$$A+B=\{a+b\mid a\in A,\,b\in B\}.$$ We denote by
$\overline{\mathcal C}$ the topological closure of a set $\mathcal C\subseteq\mathbb R^N$.

Let $X$ be a smooth projective variety and $\mathbf R\in \{\mathbb Z,\mathbb Q,\mathbb R\}$. We
denote by $\Div_{\mathbf R}(X)$ the group of
$\mathbf R$-divisors on $X$, and $\sim_{\mathbf R}$ and $\equiv$ denote $\mathbf R$-linear and numerical equivalence of $\mathbb R$-divisors.
If $A=\sum a_iC_i$ and $B=\sum b_iC_i$ are two
$\mathbb{R}$-divisors on $X$,  then $\rfdown A.=\sum \rfdown a_i.C_i$ is the round-down of $A$, $\rcup A.=\sum \rcup a_i.C_i$ is the round-up of $A$,
$\{A\}=A-\rfdown A.$ is the fractional part of $A$, $\|A\|=\max\limits_i \{|a_i|\}$ is the sup-norm of $A$, and
$$
A \wedge B= \sum \min\{a_i,b_i\} C_i.
$$
Given $D\in\Div_\mathbb{R}(X)$ and $x\in X$, $\mult_x D$ is the order of vanishing of $D$ at $x$. If $S$ is a prime divisor,
$\mult_S D$ is the order of vanishing of $D$ at the generic point of $S$.

In this paper, a {\em log pair} $(X, \Delta)$ consists of a smooth variety $X$ and an $\mathbb R$-divisor $\Delta
\ge 0$.  We say that $(X,\Delta)$ is {\em log smooth} if $\Supp\Delta$ has
simple normal crossings. A projective birational morphism $f\colon\map Y.X.$ is a {\em log resolution} of the
pair $(X, \Delta)$ if $Y$ is smooth, $\Exc f$ is a divisor and the support of
$f_*^{-1}\Delta+\Exc f$ has simple normal crossings.

\begin{definition}
Let $(X,\Delta)$ be a log pair with $\rfdown \Delta.=0$. Then $(X,\Delta)$ has {\em klt\/} (respectively {\em canonical\/},
{\em terminal\/}) singularities if for every log resolution $f\colon \map Y.X.$, if we  write $E=K_Y+f^{-1}_*\Delta-f^*(K_X+\Delta)$, we have $\rcup E.\geq0$ (respectively $E\geq0$;
$E\geq0$ and $\Supp E=\Exc f$).
Note that if $(X,\Delta)$ is terminal, then for every $\mathbb R$-divisor $G$, the pair $(X,\Delta+\varepsilon G)$ is also terminal for every $0\leq\varepsilon\ll1$.
\end{definition}

The following result is standard.

\begin{lemma}\label{l_disjoint}
Let $(X,S+B)$ be a log smooth projective pair, where $S$ is a prime divisor and $B$ is a $\mathbb Q$-divisor such that $\rfdown B.=0$ and $S\nsubseteq\Supp B$.
Then there exist a log resolution
$f\colon \map Y.X.$ of $(X,S+B)$ and $\mathbb Q$-divisors $C, E\ge0$ on $Y$
with no common components, such that the components of $C$ are disjoint, $E$
is $f$-exceptional, and if $T=f^{-1}_*S$, then
$$K_Y+T+C=f^*(K_X+S+B)+E.$$
\end{lemma}
\begin{proof}
By \cite[Proposition 2.36]{KM98}, there exist a log resolution $f\colon\map Y.X.$ which is a sequence of blow-ups
along intersections of components of $B$, and $\mathbb Q$-divisors $C, E\ge0$ on $Y$
with no common components, such that the components of $C$ are disjoint, $E$ is $f$-exceptional, and
$$K_Y+C=f^*(K_X+B)+E.$$
Since $(X,S+B)$ is log smooth, it follows that if some components of $B$ intersect, then no irreducible component of their intersection is contained in
$S$. Thus $T=f^*S$, and the lemma follows.
\end{proof}

If $X$ is a smooth projective variety, and if $D$ is an integral divisor on $X$, we denote by $\bs|D|$ the base locus of $D$.
If $D$ is an $\mathbb R$-divisor on $X$, we denote
$$|D|_{\mathbf R}=\{D'\geq0\mid D\sim_{\mathbf R} D'\}\qquad\text{and}\qquad \Bs(D)=\bigcap_{D'\in|D|_\mathbb R}\Supp D',$$
and we call $\Bs(D)$ the {\em stable base locus} of $D$. We set $\Bs(D)=X$ if $|D|_\mathbb R=\emptyset$.
The following result shows that this is compatible with the usual definition, see \cite[Lemma 3.5.3]{BCHM10}.

\begin{lemma}\label{l_real-stable-bs}
Let $X$ be a smooth projective variety and let $D$ be a $\mathbb Q$-divisor.
Then $\Bs(D)=\bigcap\nolimits_q \bs|q D|$ for all $q$ sufficiently divisible.
\end{lemma}
\begin{proof}
Fix a point $x\in X\setminus\Bs(D)$. Then there exist an $\mathbb R$-divisor $F\ge 0$, real numbers $r_1,\dots,r_k$ and rational
functions $f_1,\dots,f_k\in k(X)$ such that $F=D+\sum_{i=1}^k r_i (f_i)$ and $x\notin\Supp F$.
Let $W\subseteq \Div_{\mathbb R}(X)$ be the subspace spanned by the components
of $D$ and all $(f_i)$. Let $W_0\subseteq W$ be
the subspace of divisors $\mathbb R$-linearly equivalent to zero, and note that
$W_0$ is a rational subspace of $W$.
Consider the quotient map $\pi\colon \map W.W/W_0.$. Then the set
$\{G\in \pi^{-1}(\pi(D))\mid G\ge 0\}$
is not empty as it contains $F$, and it is cut out from $W$ by rational
hyperplanes.
Thus, it contains a $\mathbb Q$-divisor $D'\ge 0$ such
that $D\sim_{\mathbb Q}D'$ and  $x\notin\Supp D'$.
\end{proof}

\begin{definition}\label{d_variouspolytopes}
Let $(X,S+\sum_{i=1}^p S_i)$ be a log smooth projective pair, where
$S$ and all $S_i$ are distinct prime divisors, let $V=\sum_{i=1}^p\mathbb R
S_i\subseteq \Div_{\mathbb R}(X)$, and let $A$ be a  $\mathbb Q$-divisor on $X$.
We define
\begin{align*}
\mathcal L(V)&=\{ B=\textstyle\sum b_iS_i\in V \mid 0\le b_i\le 1\text{ for
all }i\},\\
\mathcal E_A(V)&=\{B\in\mathcal L(V)\mid |K_X+A+B|_{\mathbb
R}\neq\emptyset\},\\
\mathcal B_A^S(V)&=\{B\in\mathcal L(V)\mid
S\nsubseteq\Bs(K_X+S+A+B)\}.
\end{align*}
\end{definition}

If $D$ is an integral divisor,  $\fix|D|$ and $\mob(D)$ denote the {\em fixed\/}
and {\em mobile\/} parts of $D$. Hence
$|D| = |\mob(D)| + \fix|D|$, and the base locus of $|\mob(D)|$ contains no divisors. More
generally, if $V$ is any linear system on $X$, $\fix(V )$ denotes the fixed divisor of $V$.
If $S$ is a prime divisor on $X$ such that $S\nsubseteq\bs|D|$, then $|D|_S$ denotes the image of the linear system $|D|$ under restriction
to $S$.

\begin{definition}\label{dl_omega}
Let $X$ be a smooth projective variety and let $S$ be a smooth prime divisor.
Let $C$ and $D$ be $\mathbb Q$-divisors on $X$ such that $|C|_{\mathbb Q}\neq\emptyset$,
$|D|_{\mathbb Q}\neq\emptyset$ and $S\nsubseteq\Bs(D)$. Then
by Lemma \ref{l_real-stable-bs}, we may define
$$
\bfix(C)=\liminf \frac 1 k \fix |kC|\qquad\text{and}\qquad
\bfix_S(D)=\liminf \frac 1 k\fix |kD|_{S}
$$
for all $k$ sufficiently divisible.
\end{definition}

\subsection{Convex geometry and Diophantine approximation}\label{subs_diophant}
\begin{definition}\label{d_polytope}
Let $\mathcal{C}\subseteq \mathbb R^N$ be a convex set.
A subset $F\subseteq\mathcal C$ is a \emph{face} of $\mathcal{C}$ if $F$ is convex, and whenever $tu+(1-t)v\in F$ for some $u,v\in\mathcal C$ and $0<t<1$, then $u,v\in F$.  Note that $\mathcal C$ is itself a face of $\mathcal C$.
We say that $x\in \mathcal{C}$ is an \emph{extreme point} of $\mathcal C$ if $\{x\}$ is a face of $\mathcal{C}$.
For $y\in\mathcal C$, the minimal face of $\mathcal C$ which contains $y$ is denoted by $\face(\mathcal C,y)$.
It is a well known fact that any compact convex set $\mathcal C\subseteq \mathbb R^N$ is the convex hull of its extreme points.

A \emph{polytope} in $\mathbb R^N$ is a compact set which is
the intersection of finitely many half spaces; equivalently, it is
the convex hull of finitely many points in $\mathbb R^N$. A polytope is
\emph{rational} if it is
an intersection of finitely many rational half spaces; equivalently, it is
the convex hull of finitely many rational points in $\mathbb R^N$.
A \emph{rational polyhedral cone} in $\mathbb R^N$ is a convex cone spanned by finitely many rational vectors.
\end{definition}

\begin{remark}\label{rem:1}
Given a smooth projective variety $X$, we often consider subspaces $V\subseteq \Div_{\mathbb R}(X)$ which
are spanned by a finite set of prime divisors. Thus, these divisors
implicitly define  an isomorphism between $V$ and $\mathbb R^N$ for some $N$.

With notation from Definition \ref{d_variouspolytopes},
$\mathcal L(V)$ is a rational polytope. Also, the set of rational points is dense in $\mathcal B_A^S(V)$. Indeed, if $B=\sum b_i S_i\in \mathcal B_A^S(V)$, then $B+ \sum\limits_{b_i<1} \varepsilon_i S_i\in \mathcal B_A^S(V)$ for all $0\le \varepsilon_i\ll 1$.
\end{remark}

\begin{lemma}\label{l_polytope}
Let $\mathcal{P}$ be a compact convex set in $\mathbb R^N$, and fix any norm $\|\cdot\|$ on $\mathbb{R}^N$. 

Then $\mathcal{P}$ is a polytope if and only if for every point $x\in\mathcal{P}$ there exists a real number $\delta=\delta(x,\mathcal P)>0$, such that for every $y\in\mathbb R^N$ with $0<\|x-y\|<\delta$, if $(x,y)\cap\mathcal P\neq\emptyset$, then $y\in\mathcal P$.
\end{lemma}
\begin{proof}
Suppose that $\mathcal{P}$ is a polytope and let $x\in \mathcal P$. Let $F_1,\dots,F_k$ be the set of all the faces of $\mathcal P$ which do not contain $x$. Then it is enough to define 
$$\delta(x,\mathcal P)=\min\{ \|x-y\|\mid y\in F_i \text{ for some }i=1,\dots,k\}.$$ 

Conversely, assume that $\mathcal P$ is not a polytope, and let $x_n$ be an infinite sequence of distinct extreme points of $\mathcal P$. Since $\mathcal P$ is compact, by passing to a subsequence we may assume that there exists $x=\lim\limits_{n\rightarrow\infty} x_n\in \mathcal P$. For any real number $\delta>0$ pick $k\in\mathbb N$ such that $0<\|x-x_k\|<\delta$, and set $x'=x+t(x_k-x)$ for some $1<t<\delta/\|x-x_k\|$. Then $0<\|x-x'\|<\delta$ and $\emptyset\neq(x,x_k)\subseteq(x,x')\cap\mathcal P$, but $x'\notin\mathcal P$ since $x_k$ is an extreme point of $\mathcal P$. This proves the lemma. 
\end{proof}

\begin{remark}\label{r_polytope}
With assumptions from Lemma \ref{l_polytope}, assume additionally that $\mathcal P$ does not contain the origin, and let $\mathcal C=\mathbb R_+\mathcal P$. Then the same proof shows that $\mathcal{C}$ is a polyhedral cone if and only if for every point $x\in\mathcal{C}$ there exists a real number $\delta=\delta(x,\mathcal C)>0$, such that for every $y\in\mathbb R^N$ with $0<\|x-y\|<\delta$, if $(x,y)\cap\mathcal C\neq\emptyset$, then $y\in\mathcal C$.
\end{remark}

\begin{lemma}\label{l_polyhedral}
Let $\mathcal P\subseteq \mathbb R^N$ be a polytope which does not contain the origin, and let $\mathcal D=\mathbb R_+\mathcal P$. Let $\Sigma\in \mathcal D\backslash\{0\}$ and let $\Sigma_m\in\mathbb{R}^N$ be a sequence of distinct points such that $\lim\limits_{m\rightarrow\infty} \Sigma_m=\Sigma$. Let $S\in \mathbb R^N\backslash\{0\}$, let $c_m\geq0$ be a bounded sequence of real numbers, and set $\Gamma_m = \Sigma_m - c_mS$. Assume that $(\Sigma,\Gamma_m)\cap\mathcal D\neq\emptyset$ for every $m\in \mathbb N$. 

Then, there exists $P_m\in [\Sigma_m,\Gamma_m]\cap \mathcal D$ for infinitely many  $m$. 
If additionally $P_m=\Gamma_m$ and $(\Sigma_m,\Gamma_m)\cap \mathcal D=\emptyset$ for all $m$, then after passing to a subsequence, we have  $\lim\limits_{m\rightarrow\infty}\Gamma_m=\Sigma$.
\end{lemma}
\begin{proof}
Fix any norm $\|\cdot\|$ on $\mathbb{R}^N$. By passing to a subsequence we may assume that there is a constant $c\geq0$ such that $c=\lim\limits_{m\rightarrow\infty}c_m$. Assume first that  $c=0$. Then $\lim\limits_{m\rightarrow\infty} \Gamma_m=\Sigma$, and since $\mathcal D$ is a polyhedral cone, Remark \ref{r_polytope} implies that $\Gamma_m\in \mathcal D$ for $m\gg 0$ and the lemma follows. 

Thus, from now on we assume that $c>0$. First we consider the case when there are infinitely many $m$ such that $\Sigma_m\in\Sigma+\mathbb R S$. Then there is a  fixed point $P_m=P\in(\Sigma,\Gamma_m)\cap\mathcal D$ for all $m\gg 0$.  Since $\lim\limits_{m\to\infty} \Sigma_{m}=\Sigma$, it follows that  $P\in [\Sigma_m,\Gamma_m]\cap\mathcal D$ for all $m$  and the first claim follows. 
For the second claim, if additionally $(\Sigma_m,\Gamma_m)\cap \mathcal D=\emptyset$, then $P=\Sigma_m$ since $P\neq\Gamma_m$ by definition of $P$, and we immediately get a contradiction.

Therefore, we may assume that $\Sigma_m\notin\Sigma+\mathbb R S$ for all $m$. By Remark \ref{r_polytope} there exist points $Q_m\in(\Sigma,\Gamma_m)\cap\mathcal D$ and a constant $0<d<c$ such that $\|Q_m-\Sigma\|=d\quad\text{for all }m\gg 0$. After passing to a subsequence, we may assume that there exists $\lim\limits_{m\rightarrow_\infty} Q_m=Q\in\mathcal D$. Note that $Q\in [\Sigma,\Sigma-cS]$. For every $m$, as $$\Gamma_m=\frac{c_m}{c}\Sigma-\frac{c_m}{c}(\Sigma-cS)+\Sigma_m,$$ 
$\Gamma_m$ belongs to the affine $2$-plane $\{t_1\Sigma+t_2(\Sigma-cS)+t_3\Sigma_m\mid t_i\in\mathbb R,t_1+t_2+t_3=1\}$, and since $d<c$, for all $m\gg0$ there exist $P_m\in [\Sigma_m,\Gamma_m]$ such that $Q_m\in [P_m,Q]$. It is easy to see that $\lim\limits_{m\rightarrow\infty} P_m=Q$ and by Remark \ref{r_polytope} it follows that $P_m\in \mathcal D$ for $m\gg 0$, as claimed. 

Now assume additionally that $P_m=\Gamma_m$ and $(\Sigma_m,\Gamma_m)\cap \mathcal D=\emptyset$, and observe that $\lim\limits_{m\rightarrow\infty} \Gamma_m=\Sigma - c S\neq \Sigma$. Denote $\Gamma=\Sigma-\frac c2 S\in \mathcal D$ and $R_m=\Gamma_m+\frac c2 S$; note that $R_m\in(\Sigma_m,\Gamma_m)$ for $m\gg0$, and thus $R_m\notin\mathcal D$. Let $\delta=\delta(\Gamma,\mathcal D)>0$, whose existence is guaranteed by Remark \ref{r_polytope}, and pick $m\gg 0$ such that $\|R_m-\Gamma\|=\|\Gamma_m-(\Sigma- cS)\|<\delta$ and $R_m\notin\mathcal{D}$. Then the  segments $[\Sigma,\Gamma_m]$ and $[R_m,\Gamma]$ intersect at a point $R_m'\neq\Gamma$, and we have $R_m'\in\mathcal D$ since $[\Gamma_m,\Sigma]=[P_m,\Sigma]\subseteq \mathcal D$. But then $R_m\in\mathcal D$ by Remark \ref{r_polytope}, a contradiction. Thus $c=0$ and $\lim\limits_{m\rightarrow\infty} \Gamma_m=\Sigma$.
\end{proof}

\begin{lemma}[Gordan's Lemma]\label{l_g} Let $\mathcal C\subseteq \mathbb R^N$ be a rational polyhedral cone. Then $\mathcal C\cap \mathbb Z^N$ is a finitely generated monoid.
\end{lemma}
\begin{proof}
See \cite[\S 1.2]{Fulton93}.
\end{proof}

\begin{definition}
Let $\mathcal C\subseteq \mathbb R^N$ be a convex set and let $\Phi\colon\map\mathcal C.\mathbb R.$ be a function. Then $\Phi$ is
\emph{convex} if $\Phi\big(tx+(1-t)y\big)\le t\Phi(x)+(1-t)\Phi(y)$ for any $x,y\in \mathcal C$ and any $t\in [0,1]$.
If $\mathcal C$ is a rational polytope, then $\Phi$ is \emph{rationally piecewise affine}
if there exists a finite decomposition $\mathcal C=\bigcup_{i=1}^\ell \mathcal C_i$ into rational
polytopes such that $\Phi_{|\mathcal C_i}$ is a rational affine map for all $i$. If $\mathcal C$ is a cone,
then $\Phi$ is {\em homogeneous of degree one\/} if $\Phi(tx)=t\Phi(x)$ for any $x\in\mathcal C$ and $t\in\mathbb R_+$.
\end{definition}

\begin{lemma}\label{l_rpl-extension}
Let $\mathcal H\subseteq \mathbb R^N$ be a rational affine hyperplane which does not contain the origin,
and let $\mathcal P\subseteq\mathcal H$ be a rational polytope. Let $\mathcal P_\mathbb Q=\mathcal P\cap\mathbb Q^N$, and
let $f\colon \map \mathcal P_\mathbb Q. \mathbb R.$ be a bounded convex function.
Assume that there exist $x_1,\dots,x_q\in \mathcal P_\mathbb Q$ with $f(x_i)\in \mathbb Q$ for all $i$,
and that for any $x\in \mathcal P_\mathbb Q$ there exists $(r_1,\dots,r_q)\in \mathbb R^q_+$ such that
$x=\sum r_i x_i$ and $f(x)=\sum r_if(x_i)$.

Then $f$ can be extended to a rational piecewise affine function on $\mathcal P$.
\end{lemma}
\begin{proof}
Since $\mathcal P\subseteq\mathcal H$, for any $x\in \mathcal P_\mathbb Q$ and $(r_1,\dots,r_q)\in \mathbb R^q_+$ such that
$x=\sum r_i x_i$, we have $\sum r_i=1$.
Pick $C\in\mathbb Q_+$ such that $-C\leq f(x)\leq C$ for all $x\in\mathcal P_\mathbb Q$.

Let $\mathcal Q\subseteq\mathbb R^{N+1}$ be the convex hull of all the points
$\big(x_i,f(x_i)\big)$ and $(x_i,C)$, and set $\mathcal Q'=\{(x,y)\in \mathcal P_\mathbb Q\times\mathbb R\mid f(x)\leq y\leq C\}$.
Since $f$ is convex, and all $\big(x_i,f(x_i)\big)$ and $(x_i,C)$  are contained in $\mathcal Q'$, it follows that $\mathcal Q\cap\mathbb Q^{N+1}\subseteq\mathcal Q'$.
Now, fix $(u,v)\in \mathcal Q'$. Then there exists $t\in[0,1]$ such that $v=t f(u)+(1-t)C$, and as $u\in\mathcal P_\mathbb Q$, there
exist $r_i\in \mathbb R_+$ such that $\sum r_i=1$,
$u=\sum r_i x_i$ and $f(u)=\sum r_if(x_i)$. Therefore
$$\textstyle(u,v)=\sum t r_i\big(x_i,f(x_i)\big)+\sum (1-t)r_i (x_i,C),$$
and hence $(u,v)\in \mathcal Q$. This yields 
$\mathcal Q \cap {\Bbb Q}^{N+1} =\mathcal Q'\cap  {\Bbb Q}^{N+1}$, and in particular $\mathcal Q=\overline{\mathcal Q'}$.

Define $F\colon\map \mathcal P.[-C,C].$ as
$$F(x)=\min \{ y\in [-C,C] \mid (x,y)\in \mathcal Q \}.$$
Then $F$ extends $f$, and it is rational piecewise affine as $\mathcal Q$ is a rational polytope.
\end{proof}

We use the following result from Diophantine approximation.

\begin{lemma}\label{l_diophant}
Let $\|\cdot\|$ be a norm on $\mathbb R^N$, let  $\mathcal P\subseteq \mathbb
R^N$ be a rational polytope and let $x\in \mathcal P$.
Fix a positive integer $k$ and a positive real
number $\varepsilon$.

Then there are finitely many
$x_i\in \mathcal P$ and positive integers $k_i$ divisible by $k$, such
that $k_i
x_i/k$ are integral, $\|x-x_i\|<\varepsilon/k_i$, and $x$
is a convex linear combination of $x_i$.
\end{lemma}
\begin{proof}
See \cite[Lemma 3.7.7]{BCHM10}.
\end{proof}

\subsection{Nakayama-Zariski decomposition}

We need several definitions and results from \cite{Nakayama04}.

\begin{definition}
Let $X$ be a smooth projective variety, let $A$ be an ample $\mathbb{R}$-divisor,
and let $\Gamma$ be a prime divisor. If $D\in\Div_\mathbb{R}(X)$ is a big divisor, define
$$
o_\Gamma(D)=\inf\{\mult_\Gamma D'\mid D'\in|D|_\mathbb{R}\}.
$$
If $D\in\Div_\mathbb{R}(X)$ is pseudo-effective, set
$$
\sigma_\Gamma(D)=\lim_{\varepsilon\to 0} o_\Gamma(D+\varepsilon A)\qquad\text{and}\qquad \textstyle N_\sigma(D)=\sum_\Gamma\sigma_\Gamma(D)\cdot\Gamma,
$$
where the sum runs over all prime divisors $\Gamma$ on $X$.
\end{definition}

\begin{lemma}\label{d_sigma}
Let $X$ be a smooth projective variety, let $A$ be an ample $\mathbb{R}$-divisor,
let $D$ be a pseudo-effective $\mathbb R$-divisor, and let $\Gamma$ be a prime divisor.
Then $\sigma_\Gamma(D)$ exists as a limit, it is independent of the choice of $A$, it depends only on the numerical
equivalence class of $D$, and $\sigma_\Gamma(D)=o_\Gamma(D)$ if
$D$ is big. The function $\sigma_\Gamma$ is homogeneous of degree one, convex and lower semi-continuous on the cone of
pseudo-effective divisors on $X$, and it is continuous on the cone of big divisors. For every pseudo-effective $\mathbb R$-divisor $E$ we have $\sigma_\Gamma(D)=\lim\limits_{\varepsilon\to 0} \sigma_\Gamma(D+\varepsilon E)$.

Furthermore, $N_\sigma(D)$ is an $\mathbb R$-divisor on $X$, $D-N_\sigma(D)$ is pseudo-effective, and for any $\mathbb R$-divisor $0\le F\le N_\sigma(D)$
we have $N_\sigma(D-F)=N_\sigma(D)-F$.
\end{lemma}
\begin{proof}
See \cite[\S III.1]{Nakayama04}.
\end{proof}

\begin{remark}\label{r_sigmabounded}
Let $X$ be a smooth projective variety, let $D_m$ be a sequence of pseudo-effective $\mathbb R$-divisors which converge to an $\mathbb R$-divisor $D$, and let $\Gamma$ be a prime divisor on $X$. Then the sequence $\sigma_\Gamma(D_m)$ is bounded. Indeed, pick $k\gg0$ such that $D-k\Gamma$ is not pseudo-effective, and assume that $\sigma_\Gamma(D_m)>k$ for infinitely many $m$. Then $D_m-k\Gamma$ is pseudo-effective for infinitely many $m$ by Lemma \ref{d_sigma}, a contradiction.
\end{remark}

\begin{remark}\label{r_sigmablowup}
Let $X$ be a smooth projective variety, let $D$ be a pseudo-effective $\mathbb R$-divisor, let $A$ be an ample $\mathbb R$-divisor, and let $x\in X\setminus\bigcup_{\varepsilon>0}\Bs(D+\varepsilon A)$. Let $f\colon Y\rightarrow X$ be the blowup of $X$ along $x$ with the exceptional divisor $E$. Then $\sigma_E(f^*D)=0$. To see this, observe that $E\nsubseteq\Bs(f^*D+\varepsilon f^*A)$, and thus $o_E(f^*D+\varepsilon f^*A)=0$. Letting $\varepsilon\rightarrow0$, we conclude by Lemma \ref{d_sigma}.
\end{remark}

\begin{lemma}\label{l_nakayama}
Let $X$ be a smooth projective variety, let
$D$ be a pseudo-effective $\mathbb{R}$-divisor, and let $A$ be an ample
$\mathbb{Q}$-divisor.

If $D\not\equiv N_{\sigma}(D)$, then there exist a positive integer
$k$ and a positive rational number $\beta$ such that $kA$ is integral and
$$
h^0(X,\ring X.(\rfdown mD.+  kA))> \beta m \quad \text{for all}\quad m\gg 0.
$$
\end{lemma}
\begin{proof}
Replacing $D$ by $D-N_{\sigma}(D)$, we may assume that $N_{\sigma}(D)=0$.  Now apply
\cite[Theorem V.1.11]{Nakayama04}.
\end{proof}

\begin{lemma}\label{l_linindep}
Let X be a smooth projective variety, let $D$ be a pseudo-effective $\mathbb R$-di\-vi\-sor on $X$, and let
$\Gamma_1,\dots,\Gamma_\ell$ be distinct prime divisors such that $\sigma_{\Gamma_i}(D)>0$ for all $i$.

Then for any $\gamma_j\in\mathbb R_+$ we have $\sigma_{\Gamma_i}(\sum_{j=1}^\ell\gamma_j\Gamma_j)=\gamma_i$ for every $i$.
In particular, if $D\geq0$ and if $\sigma_\Gamma(D)>0$ for every component $\Gamma$ of $D$, then $D=N_\sigma(D)$.
\end{lemma}
\begin{proof}
This is \cite[Proposition III.1.10]{Nakayama04}.
\end{proof}

\begin{lemma}\label{l_sigma}
Let $X$ be a smooth projective variety and let $\Gamma$ be a
prime divisor. Let $D$ be a pseudo-effective $\mathbb R$-divisor
and let $A$ be an ample $\mathbb R$-divisor.
\begin{enumerate}
\item[(i)] If $\sigma_\Gamma(D)=0$, then $\Gamma\nsubseteq\Bs(D+A)$.
\item[(ii)] If $\sigma_\Gamma(D)>0$, then $\Gamma\subseteq\Bs(D+\varepsilon A)$ for $0<\varepsilon\ll1$.
\end{enumerate}
\end{lemma}
\begin{proof}
For (i), note that $\sigma_\Gamma(D+\frac12 A)\leq\sigma_\Gamma(D)=0$.
By Lemma \ref{d_sigma} there exists $0\leq D'\sim_{\mathbb R}D+\frac 1 2 A$ such that $\gamma=\mult_\Gamma D'\ll1$,
and in particular $\frac 1 2A+\gamma\Gamma$ is ample. Pick $A'\sim_{\mathbb R} \frac 1 2 A+\gamma \Gamma$
such that $A'\geq0$ and $\mult_\Gamma A'=0$. Then
$$D+A\sim_{\mathbb R} D' -\gamma \Gamma + A'\ge 0 \qquad\text{and}\qquad \mult_\Gamma (D'-\gamma \Gamma +A')=0.$$
This proves the first claim.
The second claim follows from $0<\sigma_\Gamma(D)=\lim\limits_{\varepsilon\to 0} o_\Gamma(D+\varepsilon A)$, since then
$o_\Gamma(D+\varepsilon A)>0$ for $0<\varepsilon\ll1$.
\end{proof}

\subsection{Divisorial rings}
Now we establish properties of finite generation of (divisorial) graded rings that we use in the paper.

\begin{definition}
Let $X$ be a smooth projective variety and let $\mathcal S\subseteq\Div_{\mathbb Q}(X)$ be a finitely generated
monoid. Then
$$R(X,\mathcal{S})=\bigoplus_{D\in\mathcal{S}}H^0\big(X,\ring X. \big(\rfdown D.)\big)$$
is a {\em divisorial $\mathcal S$-graded ring\/}.
If $D_1,\dots,D_\ell$ are generators of $\mathcal S$ and if
$D_i\sim_{\mathbb Q} k_i(K_X+\Delta_i)$, where $\Delta_i\geq0$ and $k_i\in\mathbb Q_+$ for every $i$, the algebra
$R(X,\mathcal S)$ is an {\em adjoint ring associated to $\mathcal S$\/}; furthermore,
the {\em adjoint ring associated to the sequence $D_1,\dots,D_\ell$\/} is
$$
R(X;D_1,\dots,D_\ell)=\bigoplus_{(m_1,\dots,m_\ell)\in\mathbb
N^\ell}H^0\big(X,\ring X. (\rfdown\textstyle\sum m_iD_i.)\big).$$
Note that then there is a natural projection map
$R(X;D_1,\dots,D_\ell)\longrightarrow R(X,\mathcal{S})$.

If $\mathcal C\subseteq\Div_{\mathbb R}(X)$ is a rational polyhedral cone, then
Lemma \ref{l_g} implies that $\mathcal S=\mathcal C\cap\Div(X)$
is a finitely generated monoid, and we define the algebra $R(X,\mathcal C)$, an
{\em adjoint ring associated to $\mathcal C$\/}, to be $R(X,\mathcal S)$.
\end{definition}

\begin{definition}
Let $(X,S+D)$ be a projective pair, where $X$ is smooth, $S$ is a smooth prime
divisor and $D\geq0$ is integral, and fix $\eta\in H^0(X,\ring X. (S))$ such that $\ddiv
\eta=S$. From the exact sequence
$$0\longrightarrow H^0(X,\ring X.(D-S))\stackrel{\cdot
\eta}{\longrightarrow}H^0(X,\ring X.(D))\stackrel{\rho_{S}}{
\longrightarrow} H^0(S,\ring S.(D))$$
we define $\rest_S H^0(X,\ring X.(D))=\im(\rho_{S})$, and for $\sigma\in
H^0(X,\ring X.(D))$, denote
$\sigma_{|S}=\rho_{S}(\sigma)$. Note that
$$
\ker(\rho_{S})=H^0(X,\ring X.(D-S))\cdot\eta,
$$
and that $\rest_S H^0(X,\ring X.(D))=0$ if and only if $S\subseteq\bs|D|$.

If $\mathcal S\subseteq\Div_{\mathbb Q}(X)$ is a monoid generated by divisors $D_1,\dots,D_\ell$, the {\em restriction of $R(X,\mathcal S)$ to
$S$\/} is the $\mathcal S$-graded ring
$$\rest_S R(X,\mathcal S)=\bigoplus_{D\in\mathcal{S}}\rest_S
H^0\big(X,\ring X.(\rfdown D.)\big),$$
and similarly for $\rest_S R(X;D_1,\dots,D_\ell)$.
\end{definition}

\begin{definition}
Let $\mathcal S\subseteq\mathbb Z^r$ be a finitely generated monoid and let $R=\bigoplus_{s\in\mathcal S}R_s$
be an $\mathcal S$-graded algebra. If $\mathcal S'\subseteq\mathcal S$ is a finitely generated submonoid,
then $R'=\bigoplus_{s\in\mathcal S'}R_s$ is a {\em Veronese subring\/} of $R$.
If there exists a subgroup $\mathbb L\subset\mathbb Z^r$ of finite index such that
$\mathcal S'=\mathcal S\cap\mathbb L$, then $R'$ is a {\em Veronese subring of finite index\/} of $R$.
\end{definition}

\begin{lemma}\label{l_gordan}
Let $\mathcal{S}\subseteq\mathbb Z^r$ be a finitely generated monoid and let $R=\bigoplus_{s\in\mathcal S}R_s$ be an $\mathcal S$-graded algebra.
Let $\mathcal S'\subseteq\mathcal S$ be a finitely generated submonoid and let $R'=\bigoplus_{s\in\mathcal S'}R_s$.
\begin{enumerate}
\item[(i)] If $R$ is finitely generated over $R_0$, then $R'$ is finitely generated over $R_0$.

\item[(ii)] If $R_0$ is Noetherian, $R'$ is a Veronese subring of finite index of $R$, and
$R'$ is finitely generated over $R_0$, then $R$ is finitely generated over $R_0$.
\end{enumerate}
\end{lemma}
\begin{proof}
See \cite[Proposition 1.2.2, Proposition 1.2.4]{ADHL10}.
\end{proof}

\begin{corollary}\label{c_qlinear}
Let $f\colon Y\longrightarrow X$ be a birational map between smooth projective varieties. Let $D_1,\dots,D_\ell\in\Div_{\mathbb Q}(X)$
and $D_1',\dots,D_\ell'\in\Div_{\mathbb Q}(Y)$, and assume that there exist positive rational numbers $r_i$ and
$f$-exceptional $\mathbb Q$-divisors $E_i\geq0$ such that $D_i'\sim_{\mathbb Q}r_if^*D_i+E_i$ for every $i$. Let $S$ be a smooth prime divisor on $X$ and let $T=f_*^{-1}S$.

Then the ring $R=R(X;D_1,\dots,D_\ell)$ is finitely generated if and only if the ring $R'=R(Y;D_1',\dots,D_\ell')$ is finitely generated, and the ring $\rest_S R$ is finitely generated if and only if the ring $\rest_T R'$ is finitely generated.
\end{corollary}
\begin{proof}
Let $k$ be a positive integer such that all $kD_i$, $kr_iD_i'$ and $kE_i$ are integral, and such that $kD_i'\sim kr_if^*D_i+kE_i$ for all $i$.
Then the rings $R(X;kr_1D_1,\dots,kr_\ell D_\ell)$ and $R(Y;kD_1',\dots,kD_\ell')$
are Veronese subrings of finite index of $R$ and $R'$, respectively, and they are both isomorphic to
$R(Y;kr_1f^*D_1+kE_1,\dots,kr_\ell f^*D_\ell+kE_\ell)$. We conclude by Lemma \ref{l_gordan}. The same argument works for restricted rings. 
\end{proof}

\begin{lemma}\label{c_projection}
Let $X$ be a smooth projective variety, let $D_1, \dots, D_\ell\in\Div_{\mathbb Q}(X)$, and denote $\mathcal C=\sum_{i=1}^\ell\mathbb R_+ D_i\subseteq \Div_{\mathbb R}(X)$.
\begin{enumerate}
\item[(i)] If $R(X,\mathcal C)$ is finitely generated, then $R(X;D_1,\dots,D_\ell)$ is finitely generated.
\item[(ii)] Let $S$ be a smooth prime divisor on $X$. If $\rest_S R(X,\mathcal C)$ is finitely generated, then $\rest_S R(X;D_1,\dots,D_\ell)$ is finitely generated.
\end{enumerate}
\end{lemma}
\begin{proof}
We only show (i), since (ii) is analogous. Let $k$ be a positive integer such that $D_i'=k D_i\in\Div(X)$ for all $i$. The monoid $\mathcal S=\sum_{i=1}^\ell\mathbb N D_i'\subseteq\Div(X)$
is a submonoid of $\mathcal C\cap\Div(X)$, and thus $R(X,\mathcal S)$ is finitely generated by Lemma \ref{l_gordan}(i).
But then $R(X;D_1',\dots,D_\ell')$ is finitely generated by \cite[Proposition 1.2.6]{ADHL10}, and finally $R(X;D_1,\dots,D_\ell)$
is finitely generated by Lemma \ref{l_gordan}(ii).
\end{proof}

A stronger version of the following result can be found in \cite{ELMNP}, see \cite[Theorem 3.5]{CL10}. Here we prove it as a consequence of Lemma \ref{l_rpl-extension}.

\begin{lemma}\label{l_fingen}
Let $X$ be a smooth projective variety and let $D_1, \dots, D_\ell\in\Div_\mathbb{Q}(X)$ be
such that $|D_i|_\mathbb{Q}\neq\emptyset$ for each $i$. Let $V\subseteq \Div_{\mathbb R}(X)$ be the subspace spanned by the components of $D_1,\dots,D_\ell$,
and let $\mathcal P\subseteq V$ be the convex hull of $D_1,\dots,D_\ell$.
Assume that the ring $R(X;D_1,\dots,D_\ell)$ is finitely generated. Then:
\begin{enumerate}
\item[(i)] $\bfix$ extends to a rational piecewise affine function on $\mathcal P$;
\item[(ii)] there exists a positive integer $k$ such that for every $D\in \mathcal P$ and every $m\in\mathbb N$, if
$\frac mk D\in\Div(X)$, then $\bfix (D)=\frac1m\fix|mD|$.
\end{enumerate}
\end{lemma}

\begin{proof}
Pick a prime divisor $S\in\Div(X)\setminus V$ and a rational function $\eta\in k(X)$ such that $\mult_S\ddiv\eta=1$. Then, setting $D_i'=D_i+\ddiv\eta\sim_\mathbb Q D_i$,
we have $\mult_S D_i'=1$ and $R(X;D_1,\dots,D_\ell)\simeq R(X;D'_1,\dots,D'_\ell)$. If $\mathcal P'\subseteq\Div_\mathbb R(X)$ is the convex hull of $D'_1,\dots,D'_\ell$,
it suffices to prove claims (i) and (ii) on $\mathcal P'$. Therefore, after replacing $D_i$ by $D'_i$, we may assume that $\mathcal P$ belongs to a rational
affine hyperplane which does not contain the origin. Denote $\mathcal P_\mathbb Q=\mathcal P\cap\Div_\mathbb Q(X)$.

Fix a prime divisor $G\in V$. For all $D\in\mathcal P_\mathbb Q$ and all $m\in\mathbb N$ sufficiently divisible, let
$\varphi_m(D)=\frac 1 m \mult_G\fix|mD|$, and set $\varphi(D)=\mult_G\bfix(D)$.
Then, in order to show (i), it suffices to prove that $\varphi$ is rational piecewise affine.

For every $D\in\mathcal P_\mathbb Q$, the ring $R(X,D)$ is finitely generated by Lemma \ref{l_gordan}(i), and so
by \cite[III.1.2]{Bourbaki89}, there exists a positive integer $d$ such that $R(X,dD)$ is generated by $H^0(X,\ring X.(dD))$. Thus
\begin{equation}\label{eq:72}
\varphi(D)=\varphi_{d}(D),\quad\text{and in particular}\quad \varphi(D)\in\mathbb Q.
\end{equation}
If $\sigma_1,\dots,\sigma_q$ are generators of $R(X;D_1,\dots,D_\ell)$,
then there are $G_i\in \mathcal P$ and $m_i\in \mathbb Q_+$ such that $\sigma_i\in H^0\big(X,\ring X. (\rfdown m_i G_i.)\big)$.
Fix $D\in\mathcal P_\mathbb Q$. Let $m$ be a sufficiently divisible positive integer such that $mD\in\sum\mathbb N D_i\cap\Div(X)$, and let
$\sigma\in H^0(X,\ring X.(mD))$ be such that 
\begin{equation}\label{eq:70}
\varphi_m(D)=\frac1m\mult_G\ddiv \sigma. 
\end{equation}
Then $\sigma$ is a polynomial in $\sigma_i$, thus
there are $\alpha_i\in \mathbb N$ such that $mD=\sum \alpha_i m_i G_i$ and
\begin{equation}\label{eq:71}
\mult_G\ddiv\sigma=\sum \alpha_i\mult_G \ddiv\sigma_i. 
\end{equation}
Denote $t_{m,i}=\frac{\alpha_im_i}m$, and note that
$\mult_G\ddiv\sigma_i\ge\varphi(m_iG_i)=m_i\varphi(G_i)$. Then by \eqref{eq:70} and \eqref{eq:71} we have
$$D=\sum t_{m,i} G_i\qquad\text{and}\qquad\varphi(D)=\inf_{m\in\mathbb N}\varphi_m(D)\ge\inf_{m\in\mathbb N}\sum t_{m,i} \varphi(G_i).$$
However, for all $t_i\in\mathbb Q_+$ with $D=\sum t_i G_i$,
by convexity we have $\sum t_i\varphi(G_i)\geq\varphi(D)$. Therefore
$$\varphi(D)=\inf \sum t_i \varphi(G_i),$$
where the infimum is taken over all $(t_1,\dots,t_q)\in \mathbb R_+^q$ such that
$D=\sum t_i G_i$. By compactness, there exists
$(r_1,\dots,r_q)\in \mathbb R_+^q$ such that $D=\sum r_i G_i$ and
$\varphi(D)=\sum r_i \varphi(G_i)$. Thus, $\varphi$ is rational piecewise affine by Lemma \ref{l_rpl-extension}, and (i) follows.

Now we show (ii). After decomposing $\mathcal P$, we may assume that $\bfix$ is rational linear on $\mathbb R_+\mathcal P$.
By Lemma \ref{l_g}, the monoid $\mathcal S=\mathbb  R_+\mathcal P\cap \Div(X)$ is finitely generated, and let $F_1,\dots,F_p$ be
its generators. By \eqref{eq:72}, there exists a positive integer $k$ such that $\bfix(F_i)=\frac1k \fix|kF_i|$ for all $i$.
Let $D\in\mathcal P\cap\Div_\mathbb Q(X)$, and let $m,\alpha_i\in\mathbb N$ be such that $\frac mk D=\sum\alpha_i F_i\in \mathcal S$. Then by (i) and by convexity we have
$$\textstyle\sum\alpha_i\bfix(F_i)=\frac mk\bfix(D)\leq\frac1k\fix|mD|\leq\frac1k \sum\alpha_i\fix|kF_i|=\sum\alpha_i\bfix(F_i),$$
and hence all inequalities are equalities. This completes the proof.
\end{proof}

The following result will be used in the proof of Theorem \ref{t_finite}.

\begin{theorem}\label{t_fujinomori}
Let $(X,\Delta)$ be a projective klt pair of dimension $n$, where $\Delta$ is a $\mathbb Q$-divisor.
Then there exist a projective klt pair $(Y,\Gamma)$ of dimension at most $n$ and positive integers $p$ and $q$ such that the divisors
$p(K_X+\Delta)$ and $q(K_Y+\Gamma)$ are integral, $K_Y+\Gamma$ is big and
$$R(X,p(K_X+\Delta))\simeq R(Y,q(K_Y+\Gamma)).$$
\end{theorem}
\begin{proof}
See \cite[Theorem 5.2]{FM00}.
\end{proof}

\section{Lifting sections}
\label{s_lifting}

In this section, we prove a slight generalization of the lifting theorem by Hacon and M\textsuperscript cKernan \cite{HM10}, see Theorem \ref{t_lifting}. 

We will need the following easy consequence of Kawamata-Viehweg vanishing:
\begin{lemma}\label{l_kv}
Let $(X,B)$ be a log smooth projective pair of dimension $n$, where
$B$ is a $\mathbb Q$-divisor such that $\rfdown B.=0$. Let $A$ be a nef and big $\mathbb Q$-divisor.
\begin{enumerate}
\item[(i)] Let $S$ be a smooth prime divisor such that $S\nsubseteq\Supp B$.
If $G\in\Div(X)$ is such that $G\sim_\mathbb Q K_X+S+A+B$,
then $|G_{|S}|=|G|_S$.
\item[(ii)] Let $f\colon\map X.Y.$ be a birational morphism to a projective
variety $Y$, and let  $U\subseteq X$ be an open set such that $f_{|U}$ is an isomorphism and $U$ intersects at most one irreducible component of $B$. Let $H'$ be a very ample
divisor on $Y$ and let $H=f^*H'$.  If $F\in\Div(X)$ is such that
$F\sim_{\mathbb Q} K_X+(n+1)H+A+B$, then
$|F|$ is basepoint free at every point of  $U$.
\end{enumerate}
\end{lemma}
\begin{proof}
Considering the exact sequence
$$
\ses {\ring X.(G-S)}.{\ring X.(G)}.{\ring S.(G)}.,
$$
Kawamata-Viehweg vanishing implies $H^1(X,\ring X.(G-S))=0$. In particular, the map
$H^0(X,\ring X.(G))\longrightarrow H^0(S,\ring S.(G))$ is surjective. This proves (i).

We prove (ii) by induction on $n$. Let $x\in U$ be a closed point, and
pick a general element $T\in |H|$  which contains $x$. Then by the assumptions on $U$, it follows that 
$(X,T+B)$ is log smooth,
and since $F_{|T}\sim_{\mathbb Q} K_T+nH_{|T}+A_{|T}+B_{|T}$, by induction
$F_{|T}$ is free at $x$.
Considering the exact sequence
$$\ses {\ring X.(F-T)}.{\ring X.(F)}.{\ring T.(F)}.,$$
Kawamata-Viehweg vanishing implies that $H^1(X,\ring
X.(F-T))=0$. In particular, the map $H^0(X,\ring X.(F))\longrightarrow
H^0(T,\ring T.(F))$
is surjective, and (ii) follows.
\end{proof}

\begin{lemma}\label{l_lift}
Let $(X,S+B)$ be a projective pair, where $X$ is smooth, $S$ is a smooth prime
divisor and $B$ is a $\mathbb Q$-divisor such that $S\nsubseteq\Supp B$. Let $A$ be a nef and big $\mathbb{Q}$-divisor on $X$.
Assume that $D\in\Div(X)$ is such that
$D \sim_{\mathbb Q} K_X+S+A+B$, and let $\Sigma\in |D_{|S}|$.
Let $\Phi\in\Div_\mathbb Q(S)$ be such that the pair $(S,\Phi)$ is klt and $B_{|S}\leq \Sigma+\Phi$.

Then $\Sigma \in |D|_S$.
\end{lemma}
\begin{proof}
Let $f\colon\map Y.X.$ be a log resolution of the pair
$(X,S+B)$, and write $T=f^{-1}_*S$. Then there are $\mathbb Q$-divisors $\Gamma\geq0$ and
$E\geq0$ on $Y$ with no
common components such that $T\nsubseteq\Supp\Gamma$,
$E$ is $f$-exceptional, and
$$
K_Y+T+\Gamma=f^*(K_X+S+B)+E.
$$
Let $C=\Gamma-E$ and
\begin{equation}\label{eq:3}
G=f^*D-\rfdown C.=f^*D-\rfdown \Gamma. + \rcup E..
\end{equation}
Then
$$
G-(K_Y+T+\{C\})\sim_{\mathbb Q} f^*(K_X+S+A+B)-(K_Y+T+C)=f^*A
$$
is nef and big, and Lemma \ref{l_kv}(i) implies that
\begin{equation}\label{eq:4}
|G_{|T}|=|G|_T.
\end{equation}
Moreover, since  $E\geq0$ is $f$-exceptional, we have
\begin{align}\label{eq:5}
|G|_T + \rfdown \Gamma._{|T} & =|f^*D-\rfdown\Gamma.+\rcup E.|_T + \rfdown \Gamma._{|T}\\
& \subseteq |f^*D+\rcup E.|_T=|f^*D|_T + \rcup E._{|T}.\notag
\end{align}

Denote $g=f_{|T}\colon\map T.S.$. Then
$$
K_T+C_{|T}=g^*(K_S+B_{|S}) \quad\text{and}\quad K_T+\Psi=g^*(K_S+\Phi),
$$
for some $\mathbb Q$-divisor $\Psi$ on $T$, and note that $\rfdown\Psi.\le 0$ since
$(S,\Phi)$ is klt. Therefore
\begin{equation}\label{eq:2}
g^*(B_{|S}-\Phi)=C_{|T}-\Psi.
\end{equation}
By assumption we have that $B_{|S}\le \Sigma + \Phi$, that $g^*\Sigma$ is integral, and that the
support of $C+T$ has normal crossings, so this together with \eqref{eq:2} gives
\begin{align*}
g^*\Sigma&\geq g^*\Sigma+\rfdown\Psi.=\rfdown g^*\Sigma+\Psi.\geq \rfdown g^*(B_{|S}-\Phi)+\Psi.\\
&=\rfdown C_{|T}.=\rfdown C._{|T}= (f^*D)_{|T} - G_{|T}.
\end{align*}
Denote
$$
R= G_{|T} - (f^*D)_{|T} + g^*\Sigma.$$
Then $R\geq0$ by the above, and $g^*\Sigma \in | (f^*D)_{|T}|$ implies $R \in |G_{|T}|=|G|_T$ by \eqref{eq:4}. Therefore $R+\rfdown \Gamma._{|T}\in |f^*D|_T + \rcup E._{|T}$ by \eqref{eq:5}, and this together with \eqref{eq:3} yields
$$g^*\Sigma=R+(f^*D)_{|T}-G_{|T} = R+\rfdown \Gamma._{|T} - \rcup E._{|T} \in
|f^*D|_T,$$
hence the claim follows.
\end{proof}

\begin{lemma}\label{l_tower}
Let $(X,S+B+D)$ be a log smooth projective pair,
where $S$ is a  prime divisor, $B$ is a $\mathbb Q$-divisor such that $\rfdown B.=0$
and $S\nsubseteq \Supp B$, and $D\geq0$ is a $\mathbb Q$-divisor such that $D$ and $S+B$ have no common components.
Let $P$ be a nef $\mathbb Q$-divisor and denote $\Delta=S+B+P$. Assume that
$$
K_X+\Delta \sim_{\mathbb{Q}} D.
$$
Let $k$ be a positive integer such that $kP$ and $kB$ are integral, and write
$\Omega=(B+P)_{|S}$.

Then there is
a very ample divisor $H$ such that for all divisors $\Sigma\in |k(K_S+\Omega)|$
and $U\in |H_{|S}|$,
and for every positive integer $l$ we have
$$
l\Sigma+U\in |lk(K_X+\Delta)+H|_S.
$$
\end{lemma}
\begin{proof}
For any $m\geq 0$, let $l_m=\rfdown \frac mk.$ and $r_m=m-l_mk\in\{0,1,\dots,k-1\}$, define $B_m=\rcup mB. - \rcup(m-1)B.$, and set $P_m=kP$ if $r_m=0$,
and otherwise $P_m=0$. Let
$$
D_m=\sum_{i=1}^m (K_X+S+P_i+B_i)=m(K_X+S)+l_m k P + \rcup mB. ,
$$
and note that $D_m$ is integral and
\begin{equation}\label{eq:8}
D_m=l_mk(K_X+\Delta)+D_{r_m}.
\end{equation}
By Serre vanishing, we can pick a very ample divisor $H$ on $X$ such that:
\begin{enumerate}
\item[(i)] $D_j+H$ is ample and basepoint free for every $0\leq j\leq k-1$,
\item[(ii)] $|D_k+H|_S=|(D_k+H)_{|S}|$.
\end{enumerate}
We claim that for all divisors $\Sigma\in |k(K_S+\Omega)|$ and $U_m\in
|(D_{r_m}+H)_{|S}|$
we have
$$
l_m\Sigma+ U_m\in|D_m+H|_S.
$$
The case $r_m=0$ immediately implies the lemma.

We prove the claim by induction on $m$. The case $m=k$ is covered by (ii).  Now
let $m>k$, and pick a rational number $0<\delta\ll1$ such that $D_{r_{m-1}}+H+\delta B_m$
is ample.
Note that $0\leq B_{m} \leq \rcup B.$, that $(X,S+B+D)$ is log smooth, and that $D$
and $S+B$ have no common components. Thus, there exists a rational number $0<\varepsilon\ll1$
such that, if we define
\begin{equation}\label{eq:7}
F=(1-\varepsilon\delta)B_m+l_{m-1}k\varepsilon D,
\end{equation}
then $(X,S+F)$ is log smooth, $\rfdown F.=0$ and $S\nsubseteq\Supp F$.
In particular, if $W$ is a general element of the free linear system
$|(D_{r_{m-1}}+H)_{|S}|$
and
\begin{equation}\label{eq:6}
\Phi=F_{|S}+(1-\varepsilon)W,
\end{equation}
then $(S,\Phi)$ is klt.

By induction, there is a divisor $\Upsilon\in |D_{m-1}+H|$ such that $S\nsubseteq\Supp\Upsilon$ and
$$
\Upsilon_{|S}= l_{m-1}\Sigma+ W.
$$
Denoting $C=(1-\varepsilon)\Upsilon+F$, by \eqref{eq:7} we have
\begin{equation}\label{eq:9}
C\sim_{\mathbb Q}(1-\varepsilon)(D_{m-1}+H)+(1-\varepsilon \delta)B_m + l_{m-1}k\varepsilon D,
\end{equation}
and \eqref{eq:6} yields
\begin{equation}\label{eq:11}
C_{|S}=(1-\varepsilon)\Upsilon_{|S}+F_{|S}\leq
l_{m-1}\Sigma+\Phi\leq(l_m\Sigma+U_m)+\Phi.
\end{equation}
By the choice of $\delta$ and since $P_m$ is nef, the $\mathbb Q$-divisor
\begin{equation}\label{eq:10}
A=\varepsilon (D_{r_{m-1}}+H+\delta B_m)+P_m
\end{equation}
is ample. Then by \eqref{eq:8}, \eqref{eq:10} and \eqref{eq:9} we have
\begin{align*}
D_m + H &= K_X+S+D_{m-1}+B_m + P_m +  H\\
  &= K_X+S+(1-\varepsilon)D_{m-1}+l_{m-1}k\varepsilon (K_X+\Delta) + \varepsilon
D_{r_{m-1}} + B_m + P_m + H\\
&\sim_{\mathbb Q} K_X+S+A+ (1-\varepsilon)D_{m-1}+ l_{m-1}k\varepsilon D+(1-\varepsilon
\delta)B_m+(1-\varepsilon)H\\
&\sim_{\mathbb Q} K_X+S+A+C,
\end{align*}
and thus $l_m\Sigma + U_m\in|D_m+H|_S$ by \eqref{eq:11} and Lemma \ref{l_lift}.
\end{proof}

\begin{theorem}\label{t_lifting}
Let $(X,S+B)$ be a log smooth projective pair,
where $S$ is a prime divisor,
and $B$ is a $\mathbb Q$-divisor such that $S\nsubseteq\Supp B$ and $\lfloor B\rfloor=0$.
Let  $A$ be an ample $\mathbb Q$-divisor on $X$ and denote $\Delta=S+A+B$.
Let $C\ge 0$ be a $\mathbb Q$-divisor on $S$ such that $(S,C)$ is canonical, and let $m$ be a
positive integer such that
$mA$, $mB$ and $mC$ are integral.

Assume that there exists a positive integer $q\gg 0$ such
that $qA$ is very ample,
$S\not\subseteq\base|qm(K_X+\Delta+\frac1m A)|$ and
$$C \le B_{|S}- B_{|S}\wedge \frac
1 {qm} \fix |qm(K_X+\Delta+\textstyle\frac1m A)|_S.$$
Then
$$|m(K_S+A_{|S}+C)|+m(B_{|S}-C)\subseteq |m(K_X+\Delta)|_S.$$
In particular, if $|m(K_S+A_{|S}+C)|\neq\emptyset$, then $|m(K_X+\Delta)|_S\neq\emptyset$, and
$$\fix|m(K_S+A_{|S}+C)|+ m(B_{|S}-C)\geq\fix|m(K_X+\Delta)|_S\geq m\bfix_S (K_X+\Delta).$$
\end{theorem}
\begin{proof}
 Let $f\colon \map Y.X.$ be a log resolution of the pair
$(X,S+B)$ and of the linear system $|qm(K_X+\Delta+\frac1m A)|$, and write
$T=f^{-1}_*S$.
Then there are $\mathbb Q$-divisors $B',E\geq0$ on $Y$ with no common components, such
that  $E$ is $f$-exceptional and
$$K_Y+T+B'=f^*(K_X+S+B)+E.$$
Note that
$$K_T+B'_{|T}=g^*(K_S+B_{|S})+E_{|T},$$
and since $(Y,T+B'+E)$ is log
smooth and $B'$ and $E$ do not have common components, it follows that
$B'_{|T}$ and $E_{|T}$ do not have common components, and in particular, $E_{|T}$
is $g$-exceptional and $g_*B'_{|T}=B_{|S}$.
Let $\Gamma=T+f^*A+B'$, and define
$$
F_q=\textstyle \frac 1 {qm} \fix|qm(K_Y+\Gamma+ \textstyle\frac1m  f^*A)|,\quad
B'_q=B' - B'\wedge F_q,\quad\Gamma_q=T+B'_q+f^*A.
$$
Since $(Y, T+B'+F_q)$ is log smooth, $\mob\big(qm(K_Y+\Gamma+\frac1m
f^*A)\big)$ is basepoint free, and $T\nsubseteq\Bs(K_Y+\Gamma+\frac1m f^*A)$,
by Bertini's theorem there exists a $\mathbb{Q}$-divisor $D\geq0$ such that
$$K_Y+\Gamma_q+\textstyle\frac1m f^*A\sim_\mathbb{Q} D,$$
the pair $(Y,T+B'_q+D)$ is log smooth, and $D$ does not contain
any component of $T+B'_q$.
Let $g=f_{|T}\colon \map T.S.$. Since $(S,C)$ is canonical,
there is a $g$-exceptional $\mathbb Q$-divisor $F\geq0$ on $T$
such that
$$K_T+C'=g^*(K_S+C)+F,$$
where $C'=g^{-1}_*C$. We claim that $C'\le B'_{q|T}$.
Assuming the claim, let us show how it implies the theorem.

By Lemma \ref{l_tower}, there exists
a very ample divisor $H$ on $Y$ such that for all divisors
$\Sigma'\in |qm(K_T+(B'_q+(1+\frac1m) f^*A)_{|T})|$ and $U\in |H_{|T}|$,
and for every positive integer $p$ we have
$$
p\Sigma'+U\in |p qm(K_Y+\Gamma_q+\textstyle\frac1m f^*A)+H|_T.
$$
Pick an $f$-exceptional $\mathbb Q$-divisor $G\geq0$ such that  $\rfdown B'+\frac 1 m G.=0$ and
$f^*A-G$  is ample. In particular, $(T,(B'+\frac 1 m G)_{|T})$ is klt.
Let $W_1\in|q(f^*A)_{|T}|$ and $W_2\in|H_{|T}|$ be general sections.
Pick a positive integer $k\gg0$ such that, if we denote $l=kq$,
$W=kW_1+ W_2$ and $\Phi=B'_{|T}+ \frac 1 m G_{|T}+ \frac 1 l W$,
then the $\mathbb Q$-divisor
\begin{equation}\label{eq:13}
A_0=\frac1m (f^*A-G)-\frac{m-1}{ml}H
\end{equation}
is ample and the pair $(T,\Phi)$ is klt.

Fix $\Sigma\in |m(K_S+A_{|S}+C)|$. Since $C'\le B'_{q|T}$ by the claim, it is easy to check that
$$
q g^*\Sigma + qm(F+B'_{q|T}-C')+  W_1\in
|qm(K_T+(B'_q+\textstyle(1+\frac1m) f^*A)_{|T})|.
$$
Then, by the choice of $H$, there exists
$\Upsilon\in |lm(K_Y+\Gamma_q+\frac1m f^*A)+H|$ such that $T\nsubseteq\Supp\Upsilon$ and
$$
\Upsilon_{|T}=lg^*\Sigma+lm(F+B'_{q|T}-C')+ W.
$$
Denoting
\begin{equation}\label{eq:14}
B_0=\frac{m-1}{ml}\Upsilon+(m-1)(\Gamma-\Gamma_q)+B' + \frac 1 m G,
\end{equation}
relations \eqref{eq:13} and \eqref{eq:14} imply
\begin{align}\label{eq:15}
m(K_Y+\Gamma)&=K_Y+T+(m-1)(K_Y+\Gamma+\textstyle\frac1m f^*A)+ \frac1m f^*A+B'\\
           &\sim_{\mathbb{Q}}K_Y+T+\textstyle\frac{m-1}{ml}
\Upsilon+(m-1)(\Gamma-\Gamma_q)
              +\frac1m f^*A-\frac{m-1}{ml}H+B'\notag\\
             &=K_Y+T+A_0+B_0.\notag
\end{align}
Noting that $\Gamma-\Gamma_q=B'-B'_q$, we have
\begin{align}\label{eq:12}
\textstyle B_{0|T} = \frac {m-1}m g^*\Sigma&+ (m-1)\big(F+B'_{q|T} -C' +
(\Gamma-\Gamma_q)_{|T}\big)\\
&\textstyle +\frac {m-1}{ml} W + B'_{|T}+\frac 1 m G_{|T}
\leq g^*\Sigma+m(F+B'_{|T}-C')+\Phi,\notag
\end{align}
and since $g^*\Sigma+m(F+B'_{|T}-C')\in|m(K_Y+\Gamma)_{|T}|$, by \eqref{eq:15}, \eqref{eq:12} and Lemma \ref{l_lift} we obtain
$$g^*\Sigma+m(F+B'_{|T}-C')\in|m(K_Y+\Gamma)|_T.$$
Pushing forward by $g$ yields $\Sigma+m(B_{|S}-C)\in |m(K_X+\Delta)|_S$ and the lemma follows.

Now we prove the claim stated above.
Since $\mob\big(qm(K_Y+\Gamma+\frac1m f^*A)\big)$ is basepoint free and
$T$ is not a component of $F_q$, it follows that
$\frac1{qm}\fix |qm(K_Y+\Gamma+\textstyle\frac1m f^*A)|_T=F_{q|T}$
and
$$
\textstyle B'_{q|T}=B'_{|T} - (B'\wedge F_q)_{|T}=
B'_{|T}-B'_{|T} \wedge \frac 1 {qm}\fix |qm(K_Y+\Gamma+\textstyle\frac1m
f^*A)|_T.$$
Furthermore, we have
$$
g_*\fix |qm(K_Y+\Gamma+\textstyle\frac1m f^*A)|_T=\fix
|qm(K_X+\Delta+\textstyle\frac1m A)|_S,
$$
so
$$
\textstyle g_*C'=C\leq B_{|S}- B_{|S}\wedge \frac 1 {qm} \fix
|qm(K_X+\Delta+\textstyle\frac1m A)|_S=g_*B'_{q|T}.
$$
Therefore $C'\leq B'_{q|T}$, since $B'_{q|T}\geq0$ and $C'=g^{-1}_*C$.
\end{proof}

We immediately obtain the lifting theorem from \cite{HM10}.

\begin{corollary}\label{c_inclusion}
Let $(X,S+B)$ be a log smooth projective pair,
where $S$ is a prime divisor,
and $B$ is a $\mathbb Q$-divisor such that $S\nsubseteq\Supp B$, $\lfloor B\rfloor=0$ and $(S,B_{|S})$ is canonical.
Let  $A$ be an ample $\mathbb Q$-divisor on $X$ and denote $\Delta=S+A+B$.
Let $m$ be a positive integer such that $mA$ and $mB$ are integral,
and such that $S\not\subseteq\base|m(K_X+\Delta)|$.
Denote $\Phi_m=B_{|S}- B_{|S}\wedge \frac1m \fix|m(K_X+\Delta)|_S$.

Then
$$|m(K_S+A_{|S}+\Phi_m)|+m(B_{|S}-\Phi_m)=|m(K_X+\Delta)|_S.$$
\end{corollary}
\begin{proof}
Since $\Phi_m\le B_{|S}- B_{|S}\wedge \frac1{qm} \fix|qm(K_X+\Delta+\frac1m A)|_S$ for any positive integer $q$,
the inclusion $|m(K_S+A_{|S}+\Phi_m)|+m(B_{|S}-\Phi_m)\subseteq|m(K_X+\Delta)|_S$
follows from Theorem \ref{t_lifting}, whereas the reverse inclusion is implied by $m(B_{|S}-\Phi_m)\leq\fix|m(K_X+\Delta)|_S$.
\end{proof}

\begin{lemma}\label{l_fix}
Let $X$ be a smooth projective variety and let $S$ be a smooth
prime divisor on $X$. Let $D$ be a $\mathbb Q$-divisor such that $S\nsubseteq \Bs(D)$, and
let $A$ be an ample $\mathbb Q$-divisor.
Then 
$$\frac 1 q\fix |q(D+A)|_S\le \bfix_S (D)$$
for any sufficiently divisible positive integer $q$.
\end{lemma}

\begin{proof}Let $P$ be a prime divisor on $S$ and let $\gamma=\mult_P \bfix_S(D)$.
It is enough to show that 
$$\mult_P \frac 1 q  \fix |q(D+A)|_S\le \gamma$$
for some sufficiently divisible positive integer $q$.

Assume first that $\gamma>0$.
Let $\varepsilon>0$ be a rational number such that $\varepsilon D + A$ is ample, and pick
a positive integer $m$ such that
$$\frac {1-\varepsilon} m\mult_P \fix|mD|_{S} \le \gamma.$$
Let $q$ be a sufficiently divisible positive integer such that the divisor
$q(\varepsilon D+A)$ is very ample, and such that $m$ divides $q(1-\varepsilon)$. Then
\begin{multline*}
\frac 1 q \mult_P \fix|q(D+A)|_{S} = \frac 1 q\mult_P \fix|q(1-\varepsilon)D+q(\varepsilon D+A)|_S \\
\le \frac 1 q \mult_P \fix |q(1-\varepsilon)D|_S\le \frac {1-\varepsilon} m \mult_P\fix |mD|_S\le \gamma.
\end{multline*}

Now assume that $\gamma=0$.
Let $n=\dim X$ and let $H$ be a very ample divisor on $X$. Pick
a positive integer $q$ such that $qA$ and $qD$ are integral, and such that
\begin{equation}\label{eq:16}
C=q A -K_X - S- nH
\end{equation}
is ample. Then there exists a $\mathbb Q$-divisor $D'\ge 0$ such that
$D'\sim_{\mathbb Q} D$, $S\not\subseteq\Supp D'$ and
$\mult_P (D'_{|S})< \frac1q$.
Let $f\colon Y\longrightarrow X$ be a log resolution of $(X,S+D')$ which is obtained as a sequence of blowups along smooth centres. Let $T=f^{-1}_*S$, and
let $E\ge 0$ be the $f$-exceptional integral divisor such that
$$K_Y+T=f^*(K_X+S)+E.$$
Then, denoting $F=qf^*(D+A) - \rfdown qf^*D'. + E$, by \eqref{eq:16} we have
$$F\sim_{\mathbb Q}qf^*A +\{qf^*D'\} +E = K_Y + T +
f^*(nH+C)+ \{qf^*D'\},$$
and in particular $|F_{|T}|=|F|_T$ by Lemma \ref{l_kv}(i).
Denote $g=f_{|T}\colon \map T.S.$ and let $P'=g^{-1}_*P$.
Since $F_{|T}\sim_{\mathbb Q}K_T+g^*(nH_{|S})+g^*(C_{|S})+ \{qf^*D'\}_{|T}$
and $g$ is an isomorphism at the
generic point of $P'$, Lemma \ref{l_kv}(ii) implies that
the base locus of $|F_{|T}|$ does not contain $P'$.
In particular, if $V\in |F|$ is a general element, then
$P\nsubseteq \Supp f_*V$.

Let $U=V+\rfdown q f^*D'.\in|qf^*(D+A)+E|$. Since $E$ is $f$-exceptional, this implies
that $f_*U\in |q(D+A)|$, and since $f_*  \rfdown q f^*D'.\le qD'$, we have
$$
\mult_{P} (f_*U)_{|S} = \mult_P (f_*V)_{|S}+\mult_P(f_*\rfdown q f^*D'.)_{|S}
\leq \mult_P q D'_{|S}<1.
$$
Thus, $\mult_{P} (f_*U)_{|S}= 0$ and the lemma follows.
\end{proof}

\section{$\mathcal B_A^S(V)$ is a rational polytope}
\label{s_main-lemma}

In this section, we prove several results which will be used in Sections \ref{s_effective} and \ref{s_finite} to deduce the non-vanishing theorem and the finite generation of the restricted  ring.

We introduce a function $\mathbf\Phi$ which is naturally related to the lifting theorem \ref{t_lifting}. More precisely, with the same notation as in Setup \ref{assumption}, given a $\mathbb Q$-divisor $B\in \mathcal B^S_A(V)$, a sufficiently divisible positive integer $m$ and a section $\Sigma\in |m(K_S+A_{|S}+\mathbf\Phi(B))|$, we can lift $\Sigma+m(B_{|S}-\mathbf \Phi(B))$ to $X$ as a section of $|m(K_X+S+A+B)|$. Using Diophantine approximation we prove that $\mathcal B^S_A(V)$ is a rational polytope and that, modulo some additional technical assumptions, the function $\mathbf\Phi(B)$ is rational piecewise linear. This latter fact implies that the restricted ring is finitely
generated: it shows that the ring in question is in fact an adjoint ring on a variety of lower dimension, thus we are able to apply induction, see Lemma \ref{l_restricted}.

In all results of this section we work in the following setup, and we write ``Setup \ref{assumption}$_n$'' to denote ``Setup \ref{assumption} in dimension $n$.''

\begin{setup}\label{assumption}
We assume Theorem \ref{t_cox}$_{n-1}$ and Theorem \ref{t_non-vanishing}$_{n-1}$. Let $(X,S+\sum_{i=1}^p S_i)$ be a log smooth projective pair of dimension $n$, where $S$ and all $S_i$ are distinct prime divisors.
Let $V=\sum_{i=1}^p \mathbb{R}S_i\subseteq \Div_{\mathbb R}(X)$, let $A$ be an ample $\mathbb{Q}$-divisor on $X$,
and let $W\subseteq \Div_{\mathbb R}(S)$ be the subspace spanned by the components of $\sum S_{i|S}$.

The set $\mathcal E_{A_{|S}}(W)$ is a rational polytope by Theorem \ref{t_non-vanishing}$_{n-1}$. If $E_1,\dots,E_d$ are its extreme points, the ring
$R(S;K_S+A_{|S}+E_1,\dots,K_S+A_{|S}+E_d)$ is finitely generated by Theorem \ref{t_cox}$_{n-1}$. Therefore, if we set 
$$\F(E)=\bfix(K_S+A_{|S}+E)$$
for a $\mathbb Q$-divisor $E\in\mathcal E_{A_{|S}}(W)$, then Lemma \ref{l_fingen} implies that $\F$ extends to a rational piecewise affine function on $\mathcal E_{A_{|S}}(W)$,
and there exists a positive integer $k$ with the property that
\begin{equation}\label{eq:25}
\F(E)=\frac1m\fix|m(K_S+A_{|S}+E)|
\end{equation}
for every $E\in \mathcal E_{A_{|S}}(W)$ and every $m\in\mathbb N$ such that $mA/k$ and $mE/k$ are integral. We define the set
$$
\mathcal F=\{E\in\mathcal E_{A_{|S}}(W)\mid E\wedge\F(E)=0\}.
$$
Then $\mathcal F$ is a subset of $\mathcal E_{A_{|S}}(W)$ defined by finitely many linear equalities and inequalities. Thus, there are finitely many rational
polytopes $\mathcal F_i$ such that $\mathcal F=\bigcup_i\mathcal F_i$. 

For a $\mathbb Q$-divisor $B\in\mathcal B^S_A(V)$, set
$$\F_S(B)=\bfix_S(K_X+S+A+B),$$
and for every positive integer $m$ such that $mA,mB$ are integral and $S\nsubseteq\bs|m(K_X+S+A+B)|$, denote
$$\Phi_m(B)= B_{|S} - B_{|S}\wedge\textstyle\frac 1 m\fix|m(K_X+S+A+B)|_S.$$
Let ${\mathbf\Phi}(B)=B_{|S}-B_{|S}\wedge\F_S(B)$, and note that ${\mathbf\Phi}(B)=\limsup \Phi_m(B)$.
\end{setup}

\begin{lemma}\label{l_phi}
Let the assumptions of Setup \ref{assumption}$_n$ hold.
If $B\in\mathcal B^S_A(V)$, then $\Phi_m(B)\in\mathcal E_{A_{|S}}(W)$ and $\Phi_m(B)\wedge\F(\Phi_m(B))=0$. In particular, if $\mathcal B^S_A(V)\neq\emptyset$, then $\mathcal F\neq\emptyset$.
\end{lemma}
\begin{proof}
Clearly $\Phi_m(B)\in\mathcal E_{A_{|S}}(W)$. For the second claim, note that since $m(B_{|S}-\Phi_m(B))\leq\fix|m(K_X+S+A+B)|_S$, we have
$$|m(K_S+A_{|S}+\Phi_m(B))|+m(B_{|S}-\Phi_m(B)) \supseteq
|m(K_X+S+A+B)|_S,$$
so
\begin{equation}\label{eq:17}
\fix |m(K_S+A_{|S}+\Phi_m(B))|+  m(B_{|S}-\Phi_m(B))
\le \fix|m(K_X+S+A+B)|_S.
\end{equation}
If $T$ is a component of $\Phi_m(B)$, then by definition
$$\mult_T \Phi_m(B)=\mult_T B_{|S} - \textstyle\frac1m \mult_T \fix|m(K_X+S+A+B)|_S,$$
which together with \eqref{eq:17} gives $\mult_T \fix|m(K_S+A_{|S}+\Phi_m(B))|=0$.
Hence $\mult_T \fix|km(K_S+A_{|S}+\Phi_m(B))|=0$
for every $k\in\mathbb N$, which implies
$$\Phi_m(B)\wedge\textstyle\frac1{km} \fix|km(K_S+A_{|S}+\Phi_m(B))|=0.$$
Letting $k\longrightarrow\infty$ yields the lemma.
\end{proof}

The main result of this section is:

\begin{theorem}\label{l_rational-polytope}
Let the assumptions of Setup \ref{assumption}$_n$ hold. Let $\mathcal G$ be a rational polytope contained
in the interior of $\mathcal L(V)$, and assume that $(S,G_{|S})$ is terminal for every $G\in\mathcal G$. Denote $\mathcal P=\mathcal G\cap\mathcal B_A^S(V)$.
Then
\begin{enumerate}
\item[(i)] $\mathcal P$ is a rational polytope,
\item[(ii)] $\mathbf\Phi$ extends to a rational piecewise affine function on $\mathcal P$, and
there exists a positive integer $\ell$ with the property that ${\mathbf\Phi}(P)=\Phi_m(P)$ for every $P\in\mathcal P$ and every positive integer $m$
such that $mP/\ell$ is integral.
\end{enumerate}
\end{theorem}

We describe briefly the strategy of the proof. The goal of the construction is to show that the \emph{subgraph} of $\mathbf\Phi$ is a finite union of \emph{convex} rational polytopes, which in itself does not have to be convex. Indeed, the function $B_{|S}\wedge \F_S(B)$ is not a convex function since it is defined as the minimum of two convex functions. This is one of the technical obstacles in the proof of Theorem \ref{l_rational-polytope}, and it is addressed in Step 3. The main point there is to show that the locus where it \emph{is} convex is a rational polytope.
This requires working in the space $\Div_\mathbb{R}(X)\times\Div_\mathbb{R}(S)$, and it is essentially dealt with in Lemma \ref{l_smallphi}. Then part (i) of Theorem \ref{l_rational-polytope} follows immediately by projecting this subgraph onto $\Div_\mathbb{R}(X)$. 

The fact that $\mathcal B^S_A(V)$ is a rational polytope is an easy, but technical consequence. The details are discussed in Corollary \ref{l_sbs}.

\begin{lemma}\label{l_smallphi}
Let the assumptions of Setup \ref{assumption}$_n$ hold.
Let $\mathcal G$ be a rational polytope contained in the interior of $\mathcal L(V)$, and assume that $(S,G_{|S})$ is terminal for every $G\in\mathcal G$. Fix a rational polytope $\mathcal F_i$ in the decomposition $\mathcal F=\bigcup_i \mathcal F_i$ and let
$$
\mathcal Q_i'=\{(G,F)\in\Div_\mathbb{Q}(X)\times\Div_\mathbb{Q}(S)\mid G\in\mathcal G\cap\mathcal B_A^S(V),F\in\mathcal F_i,F\le {\mathbf\Phi}(G)\}.
$$

Then the convex hull of $\mathcal{Q}_i'$ is a rational polytope.
\end{lemma}
\begin{proof}
{\em Step 1.} Let $\mathcal Q_i$ be the convex hull of $\mathcal Q_i'$. We first prove that $\mathcal{Q}_i'$ is dense in $\mathcal Q_i$.

To this end, fix $(G_0,F_0),(G_1,F_1)\in\mathcal Q_i'$, and for a
rational number $0\leq t\leq1$ set
$$G_t=(1-t)G_0+tG_1\in\mathcal P\quad\text{and}\quad F_t=(1-t)F_0+tF_1\in\mathcal F_i.$$
It suffices to show that $(G_t,F_t)\in\mathcal Q_i'$, i.e.\ that $F_t\le {\mathbf\Phi}(G_t)$ for every $t$. 

Let $T$ be a prime divisor in $W$. If $\mult_T F_t= 0$  for some $0<t<1$, then since $\mult_T F_0\geq0$ and $\mult_T F_1\geq0$ we must have $\mult_T F_t=0$ for all rational $t\in [0,1]$, and in particular $\mult_T F_t\le \mult_T\mathbf \Phi(G_t)$.

Otherwise, we have $\mult_T F_t>0$ for all $0<t<1$, and  it follows from the definition of $\mathcal F_i$ and by continuity of $\F$ that 
\begin{equation}\label{eq:26}
\mult_T\F(F_t)=0\quad\text{for all}\quad t\in[0,1].
\end{equation}
Let $m$ be a positive integer such that $mG_j/k$ and $mF_j/k$ are integral for $j\in\{0,1\}$.
By Lemma~\ref{l_fix}, we have $\frac 1 q \fix|q(K_X+S+A+G_j+\frac 1 m A)|_S\le \F_S(G_j)$ for any sufficiently divisible positive integer $q$. Since $\F(F_j)=\frac1m\fix |m(K_S+A_{|S}+F_j)|$ by assumption, Theorem~\ref{t_lifting} implies
$$
m\F(F_j)+m(G_{j|S}-F_j)\ge m\F_S(G_j),
$$
and therefore $\mult_T \big(G_{j|S}-\F_S(P_j)\big)\geq \mult_T F_j$ by \eqref{eq:26}.
Hence, 
$$\mult_T F_t\le \mult_T \big(G_{t|S}-\F_S(G_t)\big)\le \mult_T{\mathbf\Phi}(G_t)$$ 
for all $t$ by convexity of the function $\F_S$.\\[2mm]
{\em Step 2.} Let 
$$\mathcal C_i=\{(G,F)\in\mathcal G\times \mathcal F_i\mid F\le G_{|S}\}.$$
Note that $\mathcal C_i$ is a rational polytope and $\overline{\mathcal{Q}_i}\subseteq\mathcal C_i$. Fix a rational number $0<\varepsilon\ll1$ such that $D+\frac14 A$ is ample for any $D\in V$ with $\|D\|<\varepsilon$, and
$\varepsilon(K_X+S+A+B)+\frac14 A$ is ample for any $B\in \mathcal{L}(V)$.
In the next two steps, we prove the following:

\begin{claim}\label{claim2}
Suppose we are given $(B,C)\in \overline{\mathcal{Q}_i}$ and $(\Gamma,\Psi)\in \face\big(\mathcal C_i,(B,C)\big)$. Assume that
there exist a positive integer $m$ and a rational number $0<\phi\leq1$ such that $mA/k$, $m\Gamma/k$ and $m\Psi/k$ are integral, and
$\|\Gamma-B\|<\frac {\phi\varepsilon}{2m}$ and $\|\Psi-C\|<\frac {\phi\varepsilon}{2m}$.
Assume that for any prime divisor $T$ on $S$ we have
$$\mult_T(B_{|S}-C)>\phi \qquad \text{or}\qquad\mult_T (B_{|S}-C)\le \mult_T( \Gamma_{|S}-\Psi).$$

Then 
$(\Gamma,\Psi)\in\mathcal Q_i'$.
\end{claim}
\noindent {\em Step 3.}
Since $(B,C)\in \overline{\mathcal{Q}_i}$, and $\mathcal Q_i'$ is dense in $\mathcal Q_i$ by Step 1,
for every $\delta>0$ there exists
a point $(B_\delta,C_\delta)\in\mathcal Q'_i$ such that
$\|B-B_\delta\|< \frac {\delta}2$ and $\|C-C_\delta\|< \frac {\delta}2$.
Let us show that $S\nsubseteq\Bs(K_X+S+A+\Gamma+\frac1{2m} A)$, and
that for all prime divisors $T$ on $S$ and for all
$0< \delta< \frac \varepsilon m$, we have
\begin{equation}\label{eq:18}
\textstyle\mult_T  \bfix_S(K_X+S+A+\Gamma+\frac1{2m}
A)\le\mult_T\big(\Gamma_{|S}-\Psi\big)+\mult_T \F(C_\delta)+\delta.
\end{equation}

To this end, note that since $\|\Gamma-B_\delta\|\le \|\Gamma-B\|+\|B-B_\delta\|\le \frac {\varepsilon}{m}$,
the $\mathbb Q$-divisors 
$$\textstyle H=\Gamma-B_\delta+ \frac1{4m}A\qquad\text{and}\qquad G=\frac{\varepsilon}m
(K_X+S+A+B_\delta)+ \frac1{4m}A$$ are ample.
By assumption and by Lemma \ref{l_fix}, there exists a positive integer $q$ such that
$S\nsubseteq \bs|q(K_X+S+A+B_\delta)|$,
\begin{equation}\label{eq:23}
\textstyle \frac 1 q\fix |q(K_S+A_{|S}+C_\delta)|=\F(C_\delta),
\end{equation}
and
\begin{equation}\label{eq:24}
\textstyle \frac1q \fix|q(K_X+S+A+B_\delta+ H+\frac1{4m}A)|_S
\le \bfix_S (K_X+S+A+B_\delta+\frac1{4m}A).
\end{equation}
By Lemma \ref{l_fix}, there is an integer $w\gg0$ such that
$$\textstyle \frac 1{wq} \fix |wq(K_X+S+A+B_\delta+\frac 1q A)|_S\le \F_S (B_\delta),$$
so, as $(B_\delta,C_\delta)\in \mathcal Q_i'$, we have
$$\textstyle C_\delta\leq{\mathbf\Phi}(B_\delta)\le  B_{\delta|S}- B_{\delta|S}\wedge \frac
1 {wq} \fix |wq(K_X+S+A+B_\delta+\textstyle\frac1q A)|_S.
$$
Hence Theorem \ref{t_lifting} and \eqref{eq:23} imply
\begin{equation}\label{eq:22}
\F_S(B_\delta)\le 
B_{\delta|S}-C_\delta+\F(C_\delta).
\end{equation}
As $\Gamma+\textstyle \frac1{2m} A=B_\delta+ H+\frac1{4m}A$, we have $\Bs(K_X+S+A+\Gamma+\textstyle \frac1{2m} A)\subseteq\Bs(K_X+S+A+B_\delta)$,
and so $S\nsubseteq\Bs(K_X+S+A+\Gamma+\textstyle \frac1{2m} A)$. Then \eqref{eq:24} and \eqref{eq:22} yield
\begin{align*}
\bfix_S(K_X+S+A+\Gamma+\textstyle \frac1{2m} A)&\leq\textstyle\frac1q \fix
|q(K_X+S+A+B_\delta+ H+\textstyle \frac1{4m}A)|_S\\
&\le \textstyle \bfix_S \big((1-\textstyle\frac{\varepsilon}{m} )(K_X+S+A+B_\delta) + G\big)\\
&\le \big(1-\textstyle\frac{\varepsilon}{m}\big) \F_S(B_\delta)\le (1-\frac{\varepsilon}{m})(B_{\delta|S}-C_\delta) + \F(C_\delta),
\end{align*}
and since $(1-\textstyle\frac{\varepsilon}{m})\mult_T(B_{\delta|S}-C_\delta)
\le(1-\frac{\varepsilon}{m})\mult_T(B_{|S}-C)+\delta$ by assumption, to prove \eqref{eq:18} it is enough to show that
$$\big(1-\textstyle\frac \varepsilon m\big)\mult_T(B_{|S}-C)\le \mult_T(\Gamma_{|S}-\Psi).$$
This is obvious if $\mult_T(B_{|S}-C)\le\mult_T(\Gamma_{|S}-\Psi)$.
Otherwise, by assumption
$\phi<\mult_T(B_{|S}-C)\le\mult_T(\Gamma_{|S}-\Psi)+\frac{\phi\varepsilon} m$, and so 
\begin{multline*}
(1-\textstyle\frac{\varepsilon}{m})\mult_T(B_{|S}-C)
\le\mult_T(\Gamma_{|S}-\Psi)+\frac{\phi\varepsilon}{m}-\textstyle\frac{\varepsilon}{m}\mult_T(B_{|S}-C)\\
=\mult_T(\Gamma_{|S}-\Psi)-\textstyle\frac{\varepsilon}{m}\big(\mult_T(B_{|S} -C)-\phi\big)\leq\mult_T(\Gamma_{|S}-\Psi).
\end{multline*}
\mbox{}\\
{\em Step 4.} 
Having proved \eqref{eq:18}, we finish the proof of Claim \ref{claim2}. 
If $T$ is a component of $\Psi$, then $T$ is a component of $C$ as $(\Gamma,\Psi)\in \face\big(\mathcal C_i,(B,C)\big)$.
Thus $T\subseteq\Supp C_\delta$ for $\delta\ll 1$, and so $\mult_T\F(C_\delta)=0$ since $C_\delta\in \mathcal F_i$.
Hence, letting $\delta\longrightarrow0$ in \eqref{eq:18}, we get
\begin{equation}\label{eq:19}
\textstyle \Gamma_{|S}\wedge\bfix_S(K_X+S+A+\Gamma+\frac1{2m}A)\le \Gamma_{|S}-\Psi.
\end{equation}
By Lemma \ref{l_fix}, there exists a positive integer $\ell$ such that
\begin{equation}\label{eq:20}
\textstyle \frac1\ell \fix|\ell(K_X+S+A+\Gamma+\frac1m A)|_S
\le \bfix_S (K_X+S+A+\Gamma+\frac1{2m}A).
\end{equation}
Thus,  \eqref{eq:19} and \eqref{eq:20} give $\Psi\le \Gamma_{|S}-\Gamma_{|S}\wedge\frac1\ell \fix|\ell(K_X+S+A+\Gamma+\frac1m A)|_S$. Then Theorem \ref{t_lifting} implies that
$\Gamma\in \mathcal B_A^S(V)$ as $\Psi\in\mathcal E_{A_{|S}}(W)$, and furthermore, since $\F(\Psi)=\frac1m  \fix |m(K_S+A_{|S}+\Psi)|$ by assumption,
\begin{equation}\label{eq:21}
m\F(\Psi) +m(\Gamma_{|S}-\Psi)
\ge m\F_S(\Gamma).
\end{equation}
Since $\Psi\in \mathcal F_i$, we have
$\Psi\wedge \F(\Psi)=0$, so \eqref{eq:21} yields $\Gamma_{|S}-\Psi\geq\Gamma_{|S}\wedge\F_S(\Gamma)$, and finally $\Psi\le {\mathbf\Phi}(\Gamma)$. This proves Claim \ref{claim2}. \\[2mm]
{\em Step 5.}
We now show that $\mathcal Q_i$ is compact and that every extreme point of $\mathcal{Q}_i$ is rational. 

By abuse of notation, let $\|\cdot\|$ denote also the sup-norm on $\Div_\mathbb R(X)\times\Div_\mathbb R(S)$.
Fix a point $(B,C)\in\overline{\mathcal{Q}_i}$, and let $\Pi$ be the set of prime divisors $T$ on $S$
such that $\mult_T(B_{|S}-C)>0$. If $\Pi\neq\emptyset$, pick a positive rational number
$$\phi<\min\{\mult_T(B_{|S}-C)\mid T\in\Pi\}\leq1,$$
and set $\phi=1$ if $\Pi=\emptyset$.
By Lemma \ref{l_diophant}, there exist finitely many points $(\Gamma_j,\Psi_j)\in
\face\big(\mathcal C_i,(B,C)\big)$ and positive integers $m_j$ divisible by $k$, such that
$m_jA/k$, $m_j\Gamma_j/k$ and $m_j\Psi_j/k$ are integral, $(B,C)$ is a convex linear combination of all $(\Gamma_j,\Psi_j)$, and
$$\|(B,C)-(\Gamma_j,\Psi_j)\|< \frac {\phi\varepsilon}{2m_j}.$$
If $T$ is a prime divisor on $S$ such that $T\notin \Pi$, then $\mult_T(\Gamma_{j|S}-\Psi_j)=0$ as
$(\Gamma_j,\Psi_j)\in \face\big(\mathcal C_i,(B,C)\big)$, so Claim \ref{claim2} implies $(\Gamma_j,\Psi_j)\in\mathcal Q'_i$ for all $j$, hence $(B,C)\in\mathcal Q_i$. This shows that $\mathcal Q_i$ is closed and that all of its extreme points are rational.\\[2mm]
{\em Step 6.}
Finally we show that $\mathcal Q_i$ is a rational polytope.

To this end, assume for a contradiction that $\mathcal Q_i$ is not a polytope. Then, by Step 5 there exist infinitely many distinct rational extreme points $v_n=(B_n,C_n)$ of $\mathcal{Q}_i$, with $n\in \mathbb N$. Since $\mathcal{Q}_i$ is compact and
$\mathcal C_i$ is a rational polytope, by passing to a subsequence
there exist $v_\infty=(B_\infty,C_\infty) \in \mathcal Q_i$
and a positive dimensional face $\mathcal V$ of $\mathcal C_i$ such that
\begin{equation}\label{eq:32}
v_\infty=\lim\limits_{n\rightarrow\infty}
v_n\qquad\text{and}\qquad \face(\mathcal C_i,v_n)=\mathcal V \quad\text{for all }n\in
\mathbb N.
\end{equation}
In particular, $v_\infty \in \mathcal V$.
Let $\Pi_\infty$ be the set of all prime divisors $T$ on $S$
such that $\mult_T(B_{\infty|S}-C_\infty)>0$. If $\Pi_\infty\neq\emptyset$, pick a positive rational number
$$\phi<\min\{\mult_T(B_{\infty|S}-C_\infty)\mid T\in\Pi_\infty\}\leq1,$$
and set $\phi=1$ if $\Pi_\infty=\emptyset$.
Then, by Lemma \ref{l_diophant} there exist $v'_\infty\in \face(\mathcal C_i,v_\infty)$, and a positive integer $m$ divisible
by $k$, such that $\frac mk v'_\infty$ is integral and
$\|v_\infty-v'_\infty\|<\frac{\phi\varepsilon}{2m}$. As above, by Claim \ref{claim2} we have $v'_\infty\in \mathcal Q_i$. Pick $j\gg0$ so that
\begin{equation}\label{eq:33}
\|v_j-v'_\infty\|\le \|v_j-v_\infty\| + \|v_\infty-v'_\infty\|
<\frac {\phi\varepsilon} {2m},
\end{equation}
and that $\mult_T(B_{j|S}-C_j)>\phi$ if $T\in \Pi_\infty$.
Note that $v_j$ is contained in the relative interior of $\mathcal V$ by \eqref{eq:32},
and $v'_\infty\in \face(\mathcal C_i,v_\infty)\subseteq \mathcal V$. 
Therefore, there exists a positive integer $m'\gg0$ divisible by $k$, such that
$\frac{m+m'}k v_j$ is integral, and such that if we define
$$v_j'=\frac {m+m'}{m'}v_j-\frac m {m'}v'_\infty\in v_j+\mathbb
R_+(v_j-v'_\infty),$$
then $v_j'\in \mathcal V$. Note that $\frac{m'}kv_j'$ is integral,
\begin{equation}\label{eq:34}
v_j=\frac{m'}{m+m'}v_j'+\frac{m}{m+m'}v'_\infty,
\end{equation}
and
\begin{equation}\label{eq:35}
\|v_j'-v_j\|=\frac m {m'}\|v_j-v'_\infty\|< \frac {\phi\varepsilon}{2m'}
\end{equation}
by \eqref{eq:33}.
Furthermore, if $v_\infty'=(B_\infty',C_\infty')$, $v_j'=(B_j',C_j')$, and if $T$ is a prime divisor on $S$ such that $T\notin \Pi_\infty$, then $\mult_T(B'_{\infty|S}-C'_\infty)=0$ as
$v_\infty'\in \face(\mathcal C_i,v_\infty)$, hence \eqref{eq:34} gives
\begin{equation}\label{eq:36}
\mult_T(B_{j|S}-C_j)=\frac{m'}{m+m'}\mult_T(B_{j|S}'-C_j')\le \mult_T(B_{j|S}'-C_j').
\end{equation}
Therefore, $v_j'\in \mathcal{Q}_i$ by \eqref{eq:35}, \eqref{eq:36} and by Claim \ref{claim2}, and since
$v_j$ belongs to the interior of $[v_j',v'_\infty]$, we have that
$v_j$ is not an extreme point of $\mathcal Q_i$. This is a contradiction which proves the lemma.
\end{proof}

Finally, we can proceed to the proof of Theorem \ref{l_rational-polytope}.

\begin{proof}[{Proof of Theorem \ref{l_rational-polytope}}]
{\em Step 1.} In this step we prove (i). For every $i$, set
$$
\mathcal Q'_i=\{(P,F)\in\Div_\mathbb{Q}(X)\times\Div_\mathbb{Q}(S)\mid P\in\mathcal P,F\in\mathcal F_i,F\le {\mathbf\Phi}(P)\},
$$
and let $\mathcal{Q}_i$ be the convex hull of $\mathcal{Q}'_i$. Then each $\mathcal Q_i$ is a rational polytope by Lemma \ref{l_smallphi}. 

Let $\mathcal P_i\subseteq V$ be the image of $\mathcal Q_i$ through the first projection, and denote $\mathcal P_\mathbb Q=\mathcal P\cap\Div_\mathbb Q(X)$. For any $P\in \mathcal P_{\mathbb Q}$ and for any sufficiently divisible positive integer $m$,
we have $\big(P,\Phi_m(P)\big)\in \bigcup_i\mathcal Q_i$ by Lemma \ref{l_phi}. Hence $P\in\bigcup_i\mathcal P_i$, and compactness implies
\begin{equation}\label{eq:1}
\big(P,{\mathbf\Phi}(P)\big)\in \bigcup\nolimits_i\mathcal Q_i.
\end{equation}
Therefore $\mathcal P_{\mathbb Q}\subseteq\bigcup_i\mathcal P_i$, and since $\mathcal P_\mathbb Q$ is dense in $\mathcal P$ by Remark \ref{rem:1}, we have $\mathcal P\subseteq\bigcup_i\mathcal P_i$. The reverse inclusion follows by the definition of the sets $\mathcal Q_i'$, and this proves (i).\\[2mm]
{\em Step 2.}
For (ii), denote $\mathcal P_S=S+\mathcal P_\mathbb Q$, and note that
$\mathcal P_S$ lies in the hyperplane $S+V\subseteq\mathbb R S+V$.
Fix a prime divisor $T\in W$, and consider the map ${\mathbf\Phi}_T\colon \mathcal P_S\longrightarrow [-1,0]$ defined by
$${\mathbf\Phi}_T(S+P)={-}\mult_T{\mathbf\Phi}(P)\quad\text{for every}\quad P\in\mathcal P_\mathbb Q.$$
Let $\mathcal R_T$ be the closure of the set
$$\mathcal R'_T=\{S+P\in \mathcal P_S\mid {\mathbf\Phi}_T(S+P)\neq 0\}\subseteq\mathcal P_S.$$
Note that the condition ${\mathbf\Phi}_T(S+P)\neq 0$ implies ${\mathbf\Phi}_T(S+P)={-}\mult_T\big(P_{|S}-\F_S(P)\big)$, and since $\F_S$
is a convex map on $\mathcal P$, the set $\mathcal R_T$ is convex, and ${\mathbf\Phi}_T$ is convex on $\mathcal R_T$.\\[2mm]
{\em Step 3.}
We first show that $\mathcal R_T$ is a union of some of the sets $S+\mathcal P_i$, and therefore that it is a rational polytope since it is convex.

To this end, fix $P\in\mathcal P_\mathbb Q$ such that $S+P\in\mathcal R_T'$. Then $\big(P,{\mathbf\Phi}(P)\big)\in\mathcal Q_i$ for some $i$ by \eqref{eq:1},
and since $\mult_T{\mathbf\Phi}(P)\neq0$, we have
\begin{equation}\label{eq:28}
\mult_T C>0\quad\text{for every point }(B,C)\text{ in the relative interior of }\mathcal Q_i.
\end{equation}
Therefore, the definition of $\mathcal F$ yields
\begin{equation}\label{eq:27}
\mult_T\F(C)=0\quad\text{for all}\quad (B,C)\in\mathcal Q'_i.
\end{equation}
Now, pick $(B,C)\in\mathcal Q_i'$, and let $m$ be a positive integer such that $mB/k$ and $mC/k$ are integral.
By Lemma~\ref{l_fix}, we have $\frac 1 q \fix|q(K_X+S+A+B+\frac 1 m A)|_S\le \F_S(B)$ for any sufficiently divisible positive integer $q$. Since
$\F(C)=\frac 1 m \fix |m(K_S+A_{|S}+C)|$ by \eqref{eq:25}, 
Theorem~\ref{t_lifting} implies
$$
m\F(C) +m(B_{|S}-C) \ge m\F_S(B),
$$
and hence $\mult_T \big(B_{|S}-\F_S(B)\big)\geq \mult_T C\geq0$ by \eqref{eq:27}.

Therefore, for every $\mathbb Q$-divisor $B\in\mathcal P_i$ we have
\begin{equation}\label{eq:29}
{\mathbf\Phi}_T(S+B)={-}\mult_T \big(B_{|S}-\F_S(B)\big)\leq {-}\mult_T C.
\end{equation}
For any $\mathbb Q$-divisor $B$ in the relative interior of $\mathcal P_i$ there exists a $\mathbb Q$-divisor $C\in\mathcal F_i$ such that $(B,C)$ is in the relative interior of $\mathcal Q_i$, hence for such $B$ we have
${\mathbf \Phi}_T(S+B)\neq 0$ by \eqref{eq:28} and \eqref{eq:29}, that is $S+B\in \mathcal R_T'$. Therefore $S+\mathcal P_i\subseteq\mathcal R_T$,
and $\mathcal R_T$ is a union of some of the sets $S+\mathcal P_i$.\\[2mm]
{\em Step 4.}
Next we prove that ${\mathbf\Phi}_T$ extends to a rational piecewise
affine map on $\mathcal R_T$, and in particular that it is continuous on $\mathcal R_T$. 

To this end, let $(P_j,F_j)$ be the extreme points of all $\mathcal Q_i$ for which $S+\mathcal P_i\subseteq\mathcal R_T$.
Since $\mathcal Q_i$ is the convex hull of $\mathcal Q'_i$, it follows that $(P_j,F_j)\in \bigcup \mathcal Q'_i$,
and in particular
\begin{equation}\label{eq:30}
\mult_T F_j\le \mult_T \mathbf \Phi(P_j)= -\mathbf \Phi_T(S+P_j).
\end{equation}
Fix $P\in\mathcal P_\mathbb Q$ such that $S+P\in\mathcal R_T$.
Then $\big(P,{\mathbf\Phi}(P)\big)\in\mathcal Q_i$ for some $i$ by \eqref{eq:1}, hence
there exist $r_j\in \mathbb R_+$ such that
\begin{equation*}
\sum r_j=1\quad\text{and}\quad\big(P,{\mathbf\Phi}(P)\big)=\sum r_j(P_j,F_j).
\end{equation*}
Thus
${\mathbf\Phi}_T(S+P)={-}\mult_T {\mathbf\Phi}(P)={-}\sum r_j\mult_TF_j$, so by convexity of ${\mathbf\Phi}_T$ and by \eqref{eq:30} we have
$$\sum r_j{\mathbf\Phi}_T(S+P_j)\geq{\mathbf\Phi}_T(S+P)={-}\sum r_j\mult_TF_j\geq\sum r_j{\mathbf\Phi}_T(S+P_j).$$
Therefore ${\mathbf\Phi}_T(S+P_j)={-}\mult_TF_j\in\mathbb Q$ for any $j$, and ${\mathbf\Phi}_T(S+P)=\sum r_j{\mathbf\Phi}_T(S+P_j)$.
By Lemma \ref{l_rpl-extension}, ${\mathbf\Phi}_T$ extends to a rational piecewise affine map on $\mathcal R_T$.\\[2mm]
{\em Step 5.}
Note that $\F_S$ is  convex  on $\mathcal P$. Thus, by definition, ${\mathbf \Phi}_T$ extends to a continuous map in the relative interior of $S+\mathcal P$. This, together with Step 4, implies that ${\mathbf \Phi}_T$ extends to a rational piecewise affine map on $\mathcal P$ for every prime divisor $T\in W$, and hence so does $\mathbf \Phi$, which shows the first claim in (ii).
\\[2mm]
{\em Step 6.}
Finally, we show the second claim in (ii). From Step 5, in particular, we have ${\mathbf\Phi}(P)\in\Div_\mathbb Q(S)$ for every $P\in\mathcal P_\mathbb Q$, and by subdividing $\mathcal P$, we may assume that $\mathbf\Phi$ extends to a rational
affine map on $\mathcal P$.
By Lemma \ref{l_g}, the monoid $\mathbb R_+\mathcal P_S\cap\Div(X)$ is finitely generated, and let $q_i(S+Q_i)$ be its generators for some $q_i\in\mathbb Q_+$ and $Q_i\in\mathcal P_\mathbb Q$.
Pick a positive integer $w$ such that $wq_i{\mathbf\Phi}(Q_i)\in\Div(S)$ for every $i$, and set $\ell=wk$.

Fix $B\in\mathcal P_\mathbb Q$ and a positive integer $m$ such that $\frac m\ell B\in\Div(X)$. If $\alpha_i\in\mathbb N$ are such that $\frac m\ell(S+B)=\sum\alpha_i q_i(S+Q_i)$,
then $\frac m\ell=\sum\alpha_iq_i$, and therefore $\frac m\ell{\mathbf\Phi}(B)=\sum\alpha_i q_i{\mathbf\Phi}(Q_i)$ since $\mathbf\Phi$ is an affine map.
Hence $\frac mk{\mathbf\Phi}(B)=\sum\alpha_i wq_i{\mathbf\Phi}(Q_i)\in\Div(S)$, so $\F({\mathbf\Phi}(B))=\frac1m\fix|m(K_S+A_{|S}+{\mathbf\Phi}(B))|$ by \eqref{eq:25}.
In particular,
\begin{equation}\label{eq:31}
{\mathbf\Phi}(B)\wedge\fix|m(K_S+A_{|S}+{\mathbf\Phi}(B))|=0
\end{equation}
by the definition of $\mathcal F$, as $\big(B,{\mathbf\Phi}(B)\big)\in\bigcup_i\mathcal Q_i$ by \eqref{eq:1}. By Lemma \ref{l_fix}, there exists a positive integer $q$ such that
${\mathbf\Phi}(B)\le B_{|S}- B_{|S}\wedge \frac 1 {qm} \fix |qm(K_X+S+A+B+\textstyle\frac1m A)|_S$, and thus Theorem \ref{t_lifting} gives
\begin{multline*}
\fix|m(K_S+A_{|S}+{\mathbf\Phi}(B))|+m(B_{|S}-{\mathbf\Phi}(B))\geq\fix|m(K_X+S+A+B)|_S\\
\geq \textstyle m(B_{|S}\wedge\frac1m\fix|m(K_X+S+A+B)|_S)=m(B_{|S}-\Phi_m(B)).
\end{multline*}
This together with \eqref{eq:31} implies $\Phi_m(B)\geq{\mathbf\Phi}(B)$. But, by definition, ${\mathbf\Phi}(B)\geq\Phi_m(B)$,
and (ii) follows.
\end{proof}

\begin{corollary}\label{l_sbs}
Assume Theorem \ref{t_cox}$_{n-1}$ and Theorem \ref{t_non-vanishing}$_{n-1}$.

Let $(X,S+\sum_{i=1}^p S_i)$ be a log smooth projective pair of dimension $n$, where $S$ and all $S_i$ are distinct prime divisors.
Let $V=\sum_{i=1}^p \mathbb{R}S_i\subseteq \Div_{\mathbb R}(X)$ and let $A$ be an ample $\mathbb{Q}$-divisor on $X$.
Then:
\begin{enumerate}
\item[(i)] $\mathcal B_A^S(V)$ is a rational polytope, 
\item[(ii)] $\mathcal B_A^S(V)=\{B\in\mathcal L(V)\mid\sigma_S(K_X+S+A+B)=0\}$.

\end{enumerate}
\end{corollary}
\begin{proof}
We first prove (i). Fix $B\in\overline{\mathcal B_A^S(V)}$, and let $B_m\in\overline{\mathcal B_A^S(V)}$ be a sequence of distinct points such that $\lim\limits_{m\to\infty}B_m=B$. It is enough to show that $B\in\mathcal B_A^S(V)$, and that for some $m$ there exists $B_m'\in\mathcal B_A^S(V)$ such that $B_m\in(B,B_m')$: indeed, since $B$ is arbitrary, this implies that $\mathcal B_A^S(V)$ is closed, and that around every point there are only finitely many extreme points of $\mathcal B_A^S(V)$. The strategy of the proof is to reduce to the situation where $B$ is in the interior of $\mathcal L(V)$ and $(S,B_{|S})$ is terminal, and then to conclude by Theorem \ref{l_rational-polytope}.

Let $G\in V$ be a $\mathbb Q$-divisor such that $B-G$ is contained in the interior of $\mathcal L(V)$, and that $A+G$ is ample. Denote $B^G=B-G$, $B_m^G=B_m-G$ and $A^G=A+G$, and observe that $B^G$ and $B_m^G$ belong to $\overline{\mathcal B_{A^G}^S(V)}$ for $m\gg0$. By Lemma \ref{l_disjoint}, there exist a log resolution $f\colon\map Y.X.$ of $(X,S+B^G)$ and $\mathbb Q$-divisors $C,E\ge 0$ on $Y$ with no common components, such that the components of $C$ are disjoint, $\rfdown C.=0$, $T=f^{-1}_*S\nsubseteq\Supp C$, and
$$K_Y+T+C=f^*(K_X+S+B^G)+E.$$
We may then write
$$K_Y+T+C_m=f^*(K_X+S+B_m^G)+E_m,$$
where $C_m,E_m\ge 0$ are $\mathbb Q$-divisors on $Y$ with no common components, $\rfdown C_m.=0$, $T\nsubseteq\Supp C_m$, and note that $\lim\limits_{m\rightarrow\infty}C_m=C$. 
Let $V^\circ\subseteq \Div_{\mathbb R}(Y)$ be the subspace spanned by the components of $C$ and by all $f$-exceptional prime divisors.
Then there exists an $f$-exceptional $\mathbb Q$-divisor $F\ge 0$ such that $f^*A^G-F$ is ample, $C+F$ lies in the interior of $\mathcal L(V^\circ)$ and $(T,(C+F)_{|T})$ is terminal. Denote $A^\circ=f^*A^G-F$, $C^\circ=C+F$ and $C^\circ_m=C_m+F$ for all $m$, and observe that $C^\circ$ and $C_m^\circ$ belong to $\overline{\mathcal B_{A^\circ}^T(V^\circ)}$ for $m\gg0$.

There exists a positive rational number $\eta$ such that $(T,\Theta_{|T})$ is terminal for every $\Theta\in\mathcal L(V^\circ)$ with $\|\Theta-C^\circ\|\leq\eta$. Let $\mathcal P=\{\Theta\in\mathcal L(V^\circ)\mid\|\Theta-C^\circ\|\leq\eta\}$, and note that $\mathcal P$ is a rational polytope since we are working with the sup-norm. Thus $\mathcal P'=\mathcal P\cap\mathcal B_{A^\circ}^T(V^\circ)$ is a rational polytope by
Theorem \ref{l_rational-polytope}. In particular, it is closed, so $C^\circ$ and $C_m^\circ$ belong to $\mathcal B_{A^\circ}^T(V^\circ)$ for $m\gg0$. Therefore, $B^G=f_*C^\circ$ and $B_m^G=f_*C_m^\circ$ belong to $\mathcal B_{A^G}^S(V)$ for $m\gg0$, and hence $B,B_m\in\mathcal B_A^S(V)$. 

Since $\mathcal P'$ is a polytope, by Lemma \ref{l_polytope} for infinitely many $m$ there exist $C'_m\in \mathcal  P'$ such that $C_m^\circ\in(C^\circ,C'_m)$. Then $B_m^G\in(B^G,f_*C_m')$, and note that $f_*C_m'\in\mathcal B_{A^G}^S(V)$. If we denote $B_m'=f_*C_m'+G$, then  $B_m\in(B,B_m')$ and $S\nsubseteq\Bs(K_X+S+A+B_m')$ since $K_X+S+A+B_m'=K_X+S+A^G+f_*C_m'$. Again by Lemma \ref{l_polytope} applied to the polytope $\mathcal L(V)$ and the point $B\in\mathcal L(V)$, we can assume that $B_m'\in\mathcal L(V)$ by choosing $C_m'$ closer to $C^\circ$. Hence $B_m'\in\mathcal B_A^S(V)$, and this proves (i).

Now we prove (ii). Denoting $\mathcal Q=\{B\in\mathcal L(V)\mid\sigma_S(K_X+S+A+B)=0\}$, then
clearly $\mathcal Q\supseteq\mathcal B_A^S(V)$. For the reverse inclusion,
fix $B\in\mathcal Q$, and let $H$ be a very ample divisor such that $(X,S+\sum_{i=1}^p S_i+H)$ is log smooth and $H\nsubseteq\Supp (S+\sum_{i=1}^p S_i)$.
Let $V_H=\mathbb R H+V\subseteq \Div_{\mathbb R}(X)$, and note that $\sigma_S(K_X+S+A+B+tH)\leq\sigma_S(K_X+S+A+B)=0$ for $t>0$.
Then $B+tH\in \mathcal B_{A}^S(V_H)$ for any $0<t<1$ by Lemma \ref{l_sigma}(i), hence
$B\in \mathcal B_A^S(V_H)$ since $\mathcal B_{A}^S(V_H)$ is closed by the first part of the proof. Therefore $B\in \mathcal B_A^S(V)$.
\end{proof}

\section{Effective non-vanishing}
\label{s_effective}

In this section we prove that Theorem \ref{t_cox}$_{n-1}$ and Theorem \ref{t_non-vanishing}$_{n-1}$ imply Theorem \ref{t_non-vanishing}$_n$. We first sketch the idea of the proof. We consider the set
$$
\mathcal P_A(V)=\{B\in \mathcal L(V)\mid K_X+A+B\equiv D
\text{ for some $\mathbb R$-divisor }D\ge 0\},
$$
and prove that it is a rational polytope. Once we know that $\mathcal P_A(V)$ is a rational polytope, it is a straightforward application of the Kawamata-Viehweg vanishing to show that this set coincides with $\mathcal E_A(V)$, see Lemma \ref{l_numerical}.

In order to show that $\mathcal P_A(V)$ is a rational polytope, we first show that if an adjoint divisor $K_X+A+B$ is pseudo-effective, then it is numerically equivalent to an effective divisor, which in particular implies that the set $\mathcal P_A(V)$ is compact. This statement is usually referred to as ``non-vanishing.'' We may assume that $K_X+A+B\not\equiv N_\sigma(K_X+A+B)$, and the claim is a consequence of Corollary \ref{l_sbs}, see Lemma \ref{l_nonvanishing}.

Then, we show that $\mathcal P_A(V)$ is a polytope (rationality of this polytope is easy): we assume for contradiction that there are infinitely many exteme points $B_m$ of $\mathcal P_A(V)$, and by compactness and by passing to a subsequence we can assume that they converge to a point $B\in\mathcal P_A(V)$. We can then derive a contradiction if we can show that for some $m\gg0$ there is a point $B_m'\in\mathcal P_A(V)$ such that $B_m\in(B,B_m')$.

This is straightforward when $K_X+A+B\equiv N_\sigma(K_X+A+B)$, and the difficult case is when $K_X+A+B\not\equiv N_\sigma(K_X+A+B)$. We consider the cones
$$\mathcal C=\mathbb R_+(K_X+A+\mathcal P_A(V))\quad\text{and}\quad \mathcal C_S=\mathbb R_+(K_X+S+A+\mathcal B_A^S(V))\subseteq\mathcal C$$
for some prime divisor $S$, and note that $\mathcal C_S$ is a rational polyhedral cone by Corollary \ref{l_sbs}. We proceed in two steps.

First, non-vanishing implies that there exists an $\mathbb R$-divisor $F\ge 0$ such that $K_X+A+B\sim_{\mathbb R}F$. We then find a divisor $\Lambda\geq0$, whose support is contained in the support of $N_\sigma(K_X+A+B)$, and a positive real number $\mu$ such that $$\Sigma=(1+\mu)(K_X+A+B)-\Lambda=(1+\mu)(K_X+A+B-\textstyle\frac{1}{1+\mu}\Lambda)\in\mathcal C_S$$ 
for some prime divisor $S$ contained in the support of $F$. Then it is easy to find rational numbers $\varepsilon_m$ which converge to $1$ such that the divisors
$$\Sigma_m=(1+\mu)(K_X+A+B_m)-\varepsilon_m\Lambda=(1+\mu)(K_X+A+B_m-\textstyle\frac{\varepsilon_m}{1+\mu}\Lambda)$$ 
are pseudo-effective. Much of the proof of Theorem \ref{t_a-to-b} is devoted to proving these facts. Note that even though the divisor $B-\textstyle\frac{1}{1+\mu}\Lambda$ belongs to $\mathcal L(V)$, that is not necessarily the case with divisors $B_m-\frac{\varepsilon_m}{1+\mu}\Lambda$.

Second, in order to show that there is a point 
$B_m\in (B,B_m')$ as above,
it suffices to find a pseudo-effective divisor $\Sigma_m'$ for some $m\gg0$ such that $\Sigma_m\in(\Sigma,\Sigma_m')$, and this is done in Lemma \ref{l_nonvannishpolyhedral}, using the fact that $\mathcal C_S$ is a  rational polyhedral cone. \\[2mm]
\indent We start with the following lemma which uses ideas from  Shokurov's proof of the classical non-vanishing theorem. 
\begin{lemma}\label{l_numerical}
Let $(X,B)$ be a log smooth pair, where $B$ is
a $\mathbb Q$-divisor such that $\rfdown B.=0$. Let $A$ be a nef and big $\mathbb
Q$-divisor, and assume that there exists an $\mathbb R$-divisor $D\geq0$ such that
$K_X+A+B\equiv D$.

Then there exists a $\mathbb Q$-divisor $D'\geq0$ such that $K_X+A+B\sim_{\mathbb Q}D'$.
\end{lemma}

\begin{proof}
Let $V\subseteq\Div(X)_{\mathbb R}$ be the vector space spanned by the components of $K_X$, $A$, $B$ and $D$,
and let $\phi\colon V\longrightarrow N^1(X)_{\mathbb R}$ be the linear map sending an $\mathbb R$-divisor to its numerical class.
Since $\phi^{-1}(\phi(K_X+A+B))$
is a rational affine subspace of $V$, we can assume that $D\geq0$ is a $\mathbb Q$-divisor.

First assume that $(X,B+D)$ is log smooth. Let $m$ be a positive integer such that $m(A+B)$ and $mD$ are integral.
Denoting $F=(m-1)D+B$, $L=m(K_X+A+B)-\rfdown F.$ and $L'=mD-\rfdown F.$,
we have
$$L\equiv L'=D-B+\{F\} \equiv K_X+A+\{F\}.$$
Thus, Kawamata-Viehweg vanishing implies that $h^i(X,\ring X.(L))=h^i(X,\ring X.(L'))=0$ for all $i>0$, and since the
Euler characteristic is a numerical invariant, this yields
$h^0(X,\ring X.(L))=h^0(X,\ring X.(L'))$. As $mD$ is integral and $\rfdown B.=0$, it follows that
$$
L'= mD - \rfdown (m-1)D+B.= \rcup D-B.\ge 0,
$$
and thus $h^0\big(X,\ring X.(m(K_X+A+B))\big)=h^0(X,\ring X.(L+\rfdown F.))\geq h^0(X,\ring X.(L))=h^0(X,\ring X.(L'))>0$.

In the general case, let $f\colon \map Y.X.$ be a log resolution of $(X,B+D)$. Then there exist
$\mathbb Q$-divisors $B',E\ge 0$ with no common components such that $E$ is $f$-exceptional and $K_Y+B'=f^*(K_X+B)+E$. Therefore $K_Y+f^*A+B'\equiv f^*D+E\geq0$, so by above there exists a $\mathbb Q$-divisor $D^\circ\geq0$ such that $K_Y+f^*A+B'\sim_\mathbb Q D^\circ$. Hence $K_X+A+B\sim_\mathbb{Q}f_*D^\circ\geq0$.
\end{proof}

\begin{lemma}\label{l_highmultiplicity}
Let $X$ be a smooth projective variety of dimension $n$ and let $x\in X$.
Let $D\in\Div(X)$ and assume that $s$ is a positive integer such that $h^0(X,\ring X.(D))>{s+n\choose n}$.

Then there exists $D'\in|D|$ such that $\mult_x D'>s$.
\end{lemma}
\begin{proof}
Let $\mathfrak m\subseteq\mathcal O_X$ be the ideal sheaf of $x$. Then we have
$$
h^0(X,\mathcal O_X/\mathfrak m^{s+1})=\dim_{\mathbb C}\mathbb C[x_1,\dots,x_n]/(x_1,\dots,x_n)^{s+1}={s+n\choose n},
$$
hence $h^0(X,\ring X.(D))>h^0(X,\mathcal O_X/\mathfrak m^{s+1})$. Therefore the exact sequence
$$
0\longrightarrow \mathfrak m^{s+1}\otimes\mathcal O_X(D)\longrightarrow \mathcal O_X(D)
\longrightarrow (\mathcal O_X/\mathfrak m^{s+1})\otimes\mathcal O_X(D)\simeq\mathcal O_X/\mathfrak m^{s+1}\longrightarrow0
$$
yields $h^0\big(X,\mathfrak m^{s+1}\otimes\mathcal O_X(D)\big)>0$, so there exists a divisor $D'\in |D|$
with multiplicity at least $s+1$ at $x$.
\end{proof}

\begin{lemma}\label{l_nonvanishing}
Assume Theorem \ref{t_cox}$_{n-1}$ and Theorem \ref{t_non-vanishing}$_{n-1}$.

Let $(X,B)$ be a log smooth pair of dimension $n$, where $B$ is an $\mathbb R$-divisor such that $\rfdown B.=0$. Let $A$ be an ample
$\mathbb Q$-divisor on $X$, and assume that $K_X+A+B$ is a pseudo-effective $\mathbb R$-divisor such that $K_X+A+B\not\equiv N_\sigma(K_X+A+B)$.

Then there exists an $\mathbb  R$-divisor $F\ge 0$ such that $K_X+A+B\sim_{\mathbb R} F$.
\end{lemma}
\begin{proof}
By Lemma \ref{l_nakayama}, we have $h^0(X,\ring X.(\rfdown mk(K_X+A+B).+kA))>{nk+n \choose n}$ for any sufficiently divisible positive integers $m$ and $k$. Fix a point $x\in X\setminus\bigcup_{\varepsilon>0}\Bs(K_X+A+B+\varepsilon A)$.
Then, by Lemma \ref{l_highmultiplicity} there exists an $\mathbb R$-divisor $G\ge 0$ such that $G\sim_{\mathbb R}
mk(K_X+A+B)+kA$ and $\mult_x G>nk$, so setting $D=\frac1{mk} G$, we have
\begin{equation}\label{eq:37}
D\sim_{\mathbb{R}} K_X+A+B+\frac1m A\qquad\text{and}\qquad\mult_x D> \frac nm.
\end{equation}
For any $t\in [0,m]$, define $A_t=\frac{m-t}{m} A$ and $\Psi_t=B+t D$, so that
\begin{equation}\label{eq:39}
(1+t)(K_X+A+B) \sim_{\mathbb R}K_X+A+B+t\big(D-\textstyle\frac 1 m A\big)
=K_X+A_t+\Psi_t.
\end{equation}
Let $f\colon Y\longrightarrow X$ be a log resolution of $(X,B+D)$ constructed by first blowing up
$X$ at $x$. Then for every $t\in[0,m]$, there exist $\mathbb R$-divisors $C_t,E_t\ge 0$ with no common components
such that $E_t$ is $f$-exceptional and
\begin{equation}\label{eq:38}
K_Y+C_t=f^*(K_X+\Psi_t)+E_t.
\end{equation}
The exceptional divisor of the initial blowup gives a prime divisor $P\subseteq Y$ such that $\mult_P(K_Y-f^*K_X)=n-1$, $\mult_P f^*\Psi_t=\mult_x\Psi_t$, and $P\notin\Supp N_\sigma(f^*(K_X+A+B))$ by Remark \ref{r_sigmablowup}. Since $\mult_x \Psi_m> n$ by \eqref{eq:37}, it follows from \eqref{eq:38} that
\begin{equation}\label{eq:40}
\mult_P E_m=0\quad\text{and}\quad \mult_P C_m>1.
\end{equation}
Note that $\rfdown C_0.=0$, and denote
$$B_t=C_t-C_t\wedge N_\sigma(K_Y+f^*A_t+C_t).$$
Observe that by \eqref{eq:39} and \eqref{eq:38} we have
\begin{align*}
N_\sigma(K_Y+f^*A_t+C_t)&=N_\sigma\big(f^*(K_X+A_t+\Psi_t)\big)+E_t\\
&=(1+t)N_\sigma\big(f^*(K_X+A+B)\big)+E_t,
\end{align*}
hence $B_t$ is a continuous function in $t$.
Moreover $P\nsubseteq\Supp N_\sigma(K_Y+f^*A_m+B_m)$ by the choice of $x$ and by \eqref{eq:40}, and in particular $\mult_P B_m>1$.
Pick $0<\varepsilon \ll 1$ such that $\mult_P B_{m-\varepsilon}>1$, and let $H\ge 0$ be an
$f$-exceptional $\mathbb Q$-divisor on $Y$ such that $\rfdown B_0+H.=0$ and $f^*A_{m-\varepsilon}-H$ is ample.
Then there exists a  minimal $\lambda<m-\varepsilon$ such that $\rfdown B_\lambda+H.\neq0$, and let $S\subseteq\rfdown B_\lambda+H.$ be
a prime divisor. Since $\rfdown H.=0$, we have $S\subseteq\Supp B_\lambda$.
As $B_\lambda\wedge N_\sigma(K_Y+f^*A_\lambda+B_\lambda)=0$ by Lemma \ref{d_sigma}, we deduce that $S\nsubseteq\Supp N_\sigma(K_Y+f^*A_\lambda+B_\lambda)$.

Let $A'=f^*A_\lambda-H=f^*(\frac{m-\varepsilon-\lambda}m A)+(f^*A_{m-\varepsilon}-H)$. Then $A'$ is ample,
and since $\sigma_S(K_Y+A'+B_\lambda+H)=\sigma_S(K_Y+f^*A_\lambda+B_\lambda)=0$ by what we proved above, Corollary \ref{l_sbs} implies that
$S\nsubseteq \Bs(K_Y+A'+B_\lambda+H)=\Bs(K_Y+f^*A_\lambda+B_\lambda)$.
In particular, there exists an $\mathbb R$-divisor $F'\ge 0$ such that $K_Y+f^*A_\lambda+B_\lambda\sim_{\mathbb R}F'$,
and thus, by \eqref{eq:39} and \eqref{eq:38},
$$
K_X+\Delta\sim_{\mathbb R}\frac1{1+\lambda}f_*(K_Y+f^*A_\lambda+C_\lambda)\sim_{\mathbb R}\frac1{1+\lambda}f_*(F'+C_\lambda-B_\lambda)\geq0.
$$
This finishes the proof.
\end{proof}

\begin{lemma}\label{l_nonvannishpolyhedral}
Assume Theorem \ref{t_cox}$_{n-1}$ and Theorem \ref{t_non-vanishing}$_{n-1}$.

Let $(X,S+\sum_{i=1}^p S_i)$ be a log smooth projective pair of dimension $n$, where $S$ and the $S_i$ are distinct prime divisors.
Let $A$ be an ample $\mathbb Q$-divisor on $X$, let $W=\mathbb RS+\sum_{i=1}^p \mathbb R S_i\subseteq \Div_{\mathbb R}(X)$,
and assume $\Upsilon\in \mathcal L(W)$ and $0\leq\Sigma\in W$ are such that 
$$\mult_S \Upsilon = 1,\ \mult_S \Sigma >0,\ \sigma_S(K_X+A+\Upsilon)=0\quad\text{and}\quad K_X+A+\Upsilon\sim_{\mathbb R} \Sigma.$$
Let $\Upsilon_m\in W$ be a sequence such that  $K_X+A+\Upsilon_m$ are pseudo-effective and  $\lim\limits_{m\rightarrow\infty} \Upsilon_m=\Upsilon$.

Then for infinitely many $m$ there exist $\Upsilon_m'\in W$ such that $\Upsilon_m\in(\Upsilon,\Upsilon'_m)$ and $K_X+A+\Upsilon_m'$ are pseudo-effective.
\end{lemma}
\begin{proof}
{\em Step 1.} Denote
$$
\Sigma_m=\Sigma+\Upsilon_m-\Upsilon.
$$
Then $\Sigma_m\sim_\mathbb{R}K_X+A+\Upsilon_m$ is pseudo-effective by assumption, and hence so is
$$
\Gamma_m = \Sigma_m-\sigma_S(\Sigma_m)\cdot S.
$$
In Step 2, we will construct a rational polytope $\mathcal P$ which does not contain the origin, and the rational polyhedral cone $\mathcal D=\mathbb R_+\mathcal P\subseteq W$ such that every element of $\mathcal D$ is pseudo-effective and, after passing to a subsequence,
\begin{equation}\label{eq:41}
\Gamma_m \in \mathcal D\quad\text{for all }m> 0,\text{ and}\quad\lim_{m\rightarrow\infty} \Gamma_m=\Sigma.
\end{equation}
This immediately implies the lemma: indeed, Remark \ref{r_polytope} applied to $\mathcal D$ and to the point $\Sigma\in\mathcal D$ shows that for any $m\gg0$ there exist  $\Psi_m\in \mathcal D$ and $0<\mu_m<1$ such that
$\Gamma_m=\mu_m \Sigma + (1-\mu_m)\Psi_m$.
Then $\Psi_m$ is pseudo-effective, and thus so is the $\mathbb R$-divisor
$$\Sigma'_m = \Psi_m + \frac 1 {1-\mu_m}(\Sigma_m-\Gamma_m)=\Psi_m + \frac{\sigma_S(\Sigma_m)}{1-\mu_m}S.$$
Let $\Upsilon'_m=\frac{1}{1-\mu_m}(\Upsilon_m-\mu_m\Upsilon)\in W$. Then it is easy to check that $\Upsilon_m\in(\Upsilon,\Upsilon'_m)$ and $K_X+A+\Upsilon'_m\sim_\mathbb{R}\Sigma'_m$, and we are done.\\[2mm]
{\em Step 2.}
In this step, we construct a rational polytope $\mathcal D$ with required properties. Denote $V=\sum_{i=1}^p \mathbb R S_i\subseteq \Div_{\mathbb R}(X)$.
Let $Z=\sum\limits_{\mult_{S_j} \Upsilon=0} S_j - \sum\limits_{\mult_{S_i} \Upsilon=1} S_i $, and pick a rational number $0<\varepsilon\ll 1$ such that the $\mathbb Q$-divisor $A'=A-\varepsilon Z$ is ample. Setting $\Upsilon'=\Upsilon-S+\varepsilon Z$, we have
\begin{equation}\label{eq:44}
\Upsilon'\in\sum_{i=1}^p[\varepsilon S_i,(1-\varepsilon)S_i]
\end{equation}
and
\begin{equation}\label{eq:42}
K_X+S+A'+\Upsilon'\sim_{\mathbb R}\Sigma.
\end{equation}
By Corollary \ref{l_sbs}, $\mathcal B_{A'}^S(V)$ is a rational polytope, and denote
$$\mathcal P=\Sigma-\Upsilon'+\mathcal B^S_{A'}(V)\quad\text{and}\quad\mathcal D=\mathbb R_+\mathcal P\subseteq W.$$
Then $\mathcal P$ is a rational polytope and $\mathcal D$ is a rational polyhedral cone. Since $\sigma_S(K_X+S+A'+\Upsilon')=\sigma_S(\Sigma)=\sigma_S(K_X+A+\Upsilon)=0$ by assumption, Corollary \ref{l_sbs} implies that $\Upsilon'\in \mathcal B_{A'}^S(V)$, and therefore $\Sigma\in\mathcal P$. By the definition of $\mathcal P$ and by \eqref{eq:42}, for every $D\in\mathcal P$ there exists $B\in\mathcal B^S_{A'}(V)$ such that 
$$D=\Sigma-\Upsilon'+B\sim_\mathbb R K_X+S+A'+B.$$
Since $\mult_S\Upsilon'=\mult_S B=0$, this implies $\mult_S D = \mult_S \Sigma>0$ and, in particular, $\mathcal P$ does not contain the origin. Moreover, by the definition of $\mathcal B_{A'}^S(V)$, every $D$ is pseudo-effective, hence every element of $\mathcal D$ is pseudo-effective.

Now we prove \eqref{eq:41}. 
Let $\lambda_m=\mult_S\Gamma_m/\mult_S\Sigma\in\mathbb R$, and for every $m$ choose $0<\beta_m\ll1$ such that $\beta_m\lambda_m<1$ and $\beta_m\|\Gamma_m - \lambda_m \Sigma\|< \varepsilon$. Set $\delta_m=\beta_m\lambda_m$ and $t_m=(1-\delta_m)/(1-\delta_m+\beta_m)$, and note that $0<t_m<1$. 
We first show that 
\begin{equation}\label{eq:46}
t_m\Sigma+ (1-t_m)\Gamma_m\in\mathcal D.
\end{equation}
To this end, denote $R_m= \Upsilon'+ \beta_m\Gamma_m-\delta_m\Sigma$, and note that by the choice of $\beta_m$ and $\delta_m$ we have $\mult_S R_m = 0$. Furthermore, 
since $\|\beta_m \Gamma_m - \delta_m \Sigma\|< \varepsilon$, by \eqref{eq:44} we have $R_m\in\mathcal L(V)$. Then \eqref{eq:42} implies
\begin{align}\label{eq:48}
t_m\Sigma+ (1-t_m)\Gamma_m & =\frac{1}{1-\delta_m+\beta_m}(\Sigma-\Upsilon'+R_m)\\
& \sim_\mathbb{R} \frac{1}{1-\delta_m+\beta_m}(K_X+S+A'+R_m),\label{eq:99}
\end{align}
and observe that by assumption and by definition of $\Gamma_m$, we have
\begin{equation}\label{eq:47}
\sigma_S\big(t_m\Sigma+(1-t_m)\Gamma_m\big)\leq t_m\sigma_S(\Sigma)+(1-t_m)\sigma_S(\Gamma_m)=0.
\end{equation}
Therefore $R_m\in\mathcal B_{A'}^S(V)$ by Corollary \ref{l_sbs}, by \eqref{eq:99}, and by \eqref{eq:47}, hence \eqref{eq:48} and the definition of $\mathcal D$
imply \eqref{eq:46}.

In particular, $(\Sigma,\Gamma_m)\cap\mathcal D\neq\emptyset$. By Remark \ref{r_sigmabounded}, the sequence $\sigma_S(\Sigma_m)$ is bounded. Therefore, by Lemma \ref{l_polyhedral} and after passing to a subsequence we may assume that there exists $P_m\in [\Sigma_m,\Gamma_m]\cap \mathcal D$ for every $m$. By the definition of $\mathcal D$, there exists $r_m\in\mathbb R_+$ and $B_m\in\mathcal B_{A'}^S(V)$ such that $P_m=r_m(K_X+S+A'+B_m)$, hence $\sigma_S(P_m)=0$ by Corollary \ref{l_sbs}. Thus, the definition of $\Gamma_m$ implies $P_m=\Gamma_m$, and $(\Sigma_m,\Gamma_m)\cap\mathcal D=\emptyset$. Then \eqref{eq:41} follows from Lemma \ref{l_polyhedral} again applied to $\mathcal D$, and the proof is complete.
\end{proof}

\begin{theorem}\label{t_a-to-b}
Theorem \ref{t_cox}$_{n-1}$ and Theorem
\ref{t_non-vanishing}$_{n-1}$ imply Theorem
\ref{t_non-vanishing}$_n$.
\end{theorem}
\begin{proof}
Let
$$
\mathcal P_A(V)=\{B\in \mathcal L(V)\mid K_X+A+B\equiv D
\text{ for some $\mathbb R$-divisor }D\ge 0\}.
$$
We claim that $\mathcal P_A(V)$ is a rational polytope.
Assuming this, let
$B_1,\dots,B_q$ be the extreme points of $\mathcal P_A(V)$, and choose
$\varepsilon>0$ such that $A+\varepsilon B_i$ is ample for every $i$.
Since $K_X+A+B_i=K_X+(A+\varepsilon B_i)+(1-\varepsilon)B_i$ and $\rfdown (1-\varepsilon)B_i.=0$,
Lemma \ref{l_numerical} implies that there exist $\mathbb Q$-divisors $D_i\ge 0$ such that
$K_X+A+B_i\sim_{\mathbb Q} D_i$. Thus
$B_i\in \mathcal E_A(V)$  for every $i$, and therefore $\mathcal P_A(V)\subseteq\mathcal E_A(V)$
as $\mathcal E_A(V)$ is convex. Since obviously $\mathcal E_A(V)\subseteq\mathcal P_A(V)$,
the theorem follows.

Now we prove that $\mathcal P_A(V)$ is a rational polytope in several steps.\\[2mm]
{\em Step 1.} 
In this step we show that $\mathcal P_A(V)$ is closed. To this end, fix $B\in\overline{\mathcal P_A(V)}$ and denote $\Delta=A+B$.
In particular, $K_X+\Delta$ is pseudo-effective. If $K_X+\Delta\equiv N_\sigma(K_X+\Delta)$, then it follows immediately that $B\in \mathcal P_A(V)$. If $K_X+\Delta\not\equiv N_\sigma(K_X+\Delta)$, assume first that $\rfdown B.=0$. Then by Lemma \ref{l_nonvanishing} there exists an $\mathbb R$-divisor $F\ge 0$ such that $K_X+\Delta\sim_{\mathbb R}F$, and in particular $B\in \mathcal P_A(V)$. If $\rfdown B.\neq0$, pick a $\mathbb Q$-divisor $0\leq G\in V$ such that $A+G$ is ample and $\rfdown B-G.=0$. Then $B-G\in\mathcal P_{A+G}(V)$ by above, and hence $B\in \mathcal P_A(V)$. This implies that $\mathcal P_A(V)$ is compact.\\[2mm]
{\em Step 2.} 
We next show that $\mathcal P_A(V)$ is a polytope. Assume for contradiction that $\mathcal P_A(V)$ is not a polytope. Then there exists an infinite sequence of distinct extreme points $B_m\in \mathcal P_A(V)$. By compactness and by passing to a subsequence we can assume that there is a point $B\in\mathcal P_A(V)$ such that $\lim\limits_{m\rightarrow\infty} B_m=B$. Therefore, in order to derive a contradiction, it is enough to prove the following.
\begin{claim}\label{claim1}
Fix $B\in\mathcal P_A(V)$, and let $B_m\in \mathcal P_A(V)$ be a sequence of distinct points such that $\lim\limits_{m\rightarrow\infty} B_m=B$. Then for infinitely many $m$ there exist $B'_m\in \mathcal P_A(V)$
such that $B_m\in (B,B'_m)$.  
\end{claim}
We remark that it is enough to find \emph{one} such $m$, however the use of Lemma \ref{l_polytope} in Step 5 shows that we need this stronger version of the claim. 

We prove the claim in the following three steps. In Steps 3 and 4 we assume that $\rfdown B.=0$, and in Step 5 we reduce the general case to this one.\\[2mm]
{\em Step 3.} 
In this step we assume that $\rfdown B.=0$ and
\begin{equation}\label{eq:50}
K_X+A+B\not\equiv N_\sigma(K_X+A+B).
\end{equation}
By Lemma \ref{l_nonvanishing}, there exists an $\mathbb R$-divisor $F\ge 0$ such that
\begin{equation}\label{eq:51}
K_X+A+B\sim_{\mathbb R}F.
\end{equation}
We first prove Claim \ref{claim1} under an additional assumption that $F\in V$, and treat the general case at the end of Step 3. 

For any $t\geq0$, define
\begin{equation}\label{eq:54}
\Phi_t=B+tF,
\end{equation}
so that by \eqref{eq:51},
\begin{equation}\label{eq:61}
(1+t)(K_X+A+B) \sim_{\mathbb R}K_X+A+B+tF= K_X+A+\Phi_t.
\end{equation}
Note that $\rfdown \Phi_0.=0$ and
\begin{equation}\label{eq:52}
N_\sigma(K_X+A+\Phi_t)=(1+t)N_\sigma(K_X+A+B)=(1+t)N_\sigma(F).
\end{equation}
Thus, if we denote
\begin{equation}\label{eq:53}
\Upsilon_t=\Phi_t- \Phi_t\wedge N_\sigma(K_X+A+\Phi_t),
\end{equation}
then  $\Upsilon_t$ is a continuous function in $t$. Write $F=\sum_{j=1}^\ell f_j F_j$, where $F_j$ are prime divisors and $f_j>0$
for all $j$. Since $F\not\equiv N_\sigma(F)$ by \eqref{eq:50} and \eqref{eq:51}, Lemma \ref{l_linindep} implies that
there exists $j\in\{1,\dots,\ell\}$ such that $\sigma_{F_j}(F)=0$.  Thus, by \eqref{eq:54}, \eqref{eq:52} and \eqref{eq:53},
$$
\mult_{F_j}\Upsilon_t=\mult_{F_j}B+tf_j,
$$
so there exists a minimal $\mu>0$ such that $\rfdown \Upsilon_\mu.\neq0$. Note that $\rfdown \Upsilon_\mu.\subseteq\Supp F$, but $F_j$ is not necessarily a component of $\rfdown \Upsilon_\mu.$. Let $S\subseteq\rfdown \Upsilon_\mu.$ be a prime divisor. Observe that
\begin{equation}\label{eq:60}
\sigma_{S}(K_X+A+\Upsilon_\mu)=0
\end{equation}
by \eqref{eq:53}, and
\begin{align}\label{eq:55}
\sigma_S\big((1+\mu)F\big)&=\sigma_S(K_X+A+\Phi_\mu)=\mult_S\Phi_\mu-\mult_S\Upsilon_\mu\\
&=\mult_S B+\mu\mult_S F-1<\mu\mult_S F\notag
\end{align}
by \eqref{eq:54}, \eqref{eq:52} and \eqref{eq:53}. Let $\Sigma=(1+\mu)F- \Phi_\mu\wedge N_\sigma\big((1+\mu)F\big)$. Then we have
\begin{equation}\label{eq:57}
\Sigma\ge 0\quad\text{and}\quad K_X+A+\Upsilon_\mu\sim_\mathbb{R}\Sigma
\end{equation}
by \eqref{eq:52} and \eqref{eq:53}, and moreover,
\begin{equation}\label{eq:56}
\mult_S\Sigma\geq(1+\mu)\mult_S F-\sigma_S\big((1+\mu)F\big)>\mult_S F\geq0
\end{equation}
by \eqref{eq:55}. For every $m\in\mathbb N$, define $\Phi_{\mu,m}=B_m+\mu(F+B_m-B)$. Then
\begin{equation}\label{eq:58}
\lim\limits_{m\rightarrow\infty}\Phi_{\mu,m}=\Phi_\mu\qquad \text{and} \qquad (1+\mu)(K_X+A+B_m)\sim_\mathbb{R} K_X+A+\Phi_{\mu,m}
\end{equation}
by assumption, and by \eqref{eq:54} and \eqref{eq:51}. Let
$$\Lambda=\Phi_\mu\wedge N_\sigma(K_X+A+\Phi_\mu)\quad\text{and}\quad
\Lambda_m=\Phi_{\mu,m}\wedge\sum_{Z\subseteq\Supp\Lambda}\sigma_Z(K_X+A+\Phi_{\mu,m})\cdot Z.$$
Note that $0\leq\Lambda_m\leq N_\sigma(K_X+A+\Phi_{\mu,m})$, and therefore $K_X+A+\Phi_{\mu,m}-\Lambda_m$
is pseudo-effective by Lemma \ref{d_sigma}. By Lemma \ref{d_sigma} again, we have $\Lambda\leq\liminf\limits_{m\rightarrow\infty}\Lambda_m$, and in particular,
$\Supp\Lambda_m=\Supp\Lambda$ for $m\gg0$. Thus, there exists an increasing sequence of rational numbers $\varepsilon_m>0$ such that $\lim\limits_{m\rightarrow\infty} \varepsilon_m=1$ and
$\Lambda_m\geq\varepsilon_m\Lambda$.

Define $\Upsilon_{\mu,m}=\Phi_{\mu,m}-\varepsilon_m\Lambda$. Note that
\begin{equation}\label{eq:59}
K_X+A+\Upsilon_{\mu,m}\quad \text{is pseudo-effective}\qquad\text{and}\qquad \lim_{m\rightarrow\infty} \Upsilon_{\mu,m}=\Phi_\mu-\Lambda=\Upsilon_\mu
\end{equation}
by \eqref{eq:58} and \eqref{eq:53}. Therefore, by \eqref{eq:60}, \eqref{eq:57}, \eqref{eq:56}, \eqref{eq:59} and Lemma \ref{l_nonvannishpolyhedral}, and by passing to a subsequence, for every $m$ there exist $\Upsilon'_m\in V$ and $0<\alpha_m\ll1$ such that
$$K_X+A+\Upsilon_m'\quad \text{is pseudo-effective}\qquad\text{and}\qquad
\Upsilon_{\mu,m}=\alpha_m\Upsilon_\mu+(1-\alpha_m) \Upsilon'_m.$$
Setting $B'_m=\frac1{1-\alpha_m}(B_m-\alpha_m B)$,
we have $B_m=\alpha_m B+(1-\alpha_m)B'_m$, and an easy calculation involving \eqref{eq:54}, \eqref{eq:58} and \eqref{eq:59} shows that
$$
K_X+A+B_m'\sim_{\mathbb R}\frac1{1+\mu}\Big(K_X+A+\Upsilon'_m + \frac {\varepsilon_m -\alpha_m}{1-\alpha_m}\Lambda\Big).
$$
In particular, $K_X+A+B'_m$ is pseudo-effective for $m\gg0$.
Since $\mathcal L(V)$ is a rational polytope, Lemma \ref{l_polytope} yields $B_m'\in\mathcal L(V)$ for $m\gg0$, which proves Claim \ref{claim1} under the additional assumption that $F\in V$.

To show the general case, let $f\colon \map Y.X.$ be a log resolution of $(X,B+F)$. Then there are $\mathbb R$-divisors $C,E\ge 0$ on $Y$ with no common components and $C_m,E_m\ge 0$ on $Y$ with no common components such that $E$ and $E_m$ are $f$-exceptional and
$$K_Y+C=f^*(K_X+B)+E\quad\text{and}\quad K_Y+C_m=f^*(K_X+B_m)+E_m.$$
Note that $\lim\limits_{m\rightarrow\infty}C_m=C$. Let $G\ge 0$ be an $f$-exceptional $\mathbb Q$-divisor on $Y$ such that $A^\circ$ is ample, $\rfdown C^\circ.=0$, and $\rfdown C_m^\circ.=0$ for all $m\gg0$, where $A^\circ=f^*A-G$, $C^\circ=C+G$ and $C^\circ_m=C_m+G$. Denoting $F^\circ=f^*F+E\geq0$, we have 
$$f_*C^\circ=B,\quad f_*C_m^\circ=B_m, \quad \text{and} \quad K_Y+A^\circ+C^\circ\sim_\mathbb R F^\circ.$$ 
Let $V^\circ\subseteq\Div_{\mathbb R}(Y)$ be the vector space spanned by the components of $\sum_{i=1}^p f^{-1}_*S_i+f^{-1}_*F$ plus all exceptional prime divisors, and note that $F^\circ\in V^\circ$. By what we proved above, for infinitely many $m$ there exist $C'_m\in \mathcal  P_{A^\circ}(V^\circ)$ such that $C_m^\circ\in(C^\circ,C'_m)$. Note that $\Supp C_m'$ is a subset of $\sum_{i=1}^p f^{-1}_*S_i$ plus all exceptional prime divisors, and denote $B_m'=f_*C_m'\in\mathcal L(V)$. Then $B_m\in(B,B_m')$, and $K_X+A+B_m'=f_*(K_Y+A^\circ+C_m')$ is numerically equivalent to an effective divisor, hence $B_m'\in\mathcal P_A(V)$, finishing the proof of Claim \ref{claim1} when $\rfdown B.=0$ and $K_X+A+B\not\equiv N_\sigma(K_X+A+B)$. 
\\[2mm]
{\em Step 4.}
Now assume that $\rfdown B.=0$ and $K_X+A+B\equiv N_\sigma(K_X+A+B)$.
Let $D_m\geq0$ be $\mathbb R$-divisors such that $K_X+A+B_m\equiv D_m$.
By Lemma \ref{l_sigma}(ii), there exists an ample $\mathbb R$-divisor $H$ such that
$$\Supp N_\sigma(K_X+A+B)\subseteq\Bs(K_X+A+B+H),$$
and as $H+(K_X+A+B-D_m)\equiv H+(B-B_m)$ is ample for all $m\gg0$, by passing to a subsequence we may assume that
\begin{align}
\Supp N_\sigma(K_X+A+B)&\subseteq\Bs\big(D_m+H+(K_X+A+B-D_m)\big) \label{eq:62}\\
&\subseteq\Bs(D_m)\subseteq\Supp D_m \notag
\end{align}
for all $m$. For $m\in\mathbb N$ and $t>1$, denote $C_{m,t}=B+t (B_m-B)$, and observe that
\begin{equation}\label{eq:75}
B_m=\frac1{t}C_{m,t}+\frac {t-1}{t}B
\end{equation}
and
\begin{equation}\label{eq:76}
K_X+A+C_{m,t}\equiv tD_m-(t-1)(K_X+A+B)\equiv tD_m -(t-1) N_\sigma(K_X+A+B).
\end{equation}
Since $\mathcal L(V)$ is a polytope and $B\in\mathcal L(V)$, pick $\delta=\delta(B,\mathcal L(V))>0$ as in Lemma \ref{l_polytope}. By passing to a subsequence we may assume that $\|B_m-B\|\leq\delta/2$ for every $m$, and as $\|C_{m,t}-B\|=t\|B_m-B\|$, Lemma \ref{l_polytope} gives $C_{m,t}\in\mathcal L(V)$ for all $m$ and $1<t<2$. 

Fix $m$. By \eqref{eq:62} there exists $1<t_m<2$ such that $t_mD_m -(t_m-1)N_\sigma(K_X+A+B)\geq0$, and denote $B_m'=C_{m,t_m}$. Then \eqref{eq:76} implies $B_m'\in \mathcal P_A(V)$, and thus \eqref{eq:75} proves Claim \ref{claim1}.\\[2mm]
{\em Step 5.}
Now we treat the general case of Claim \ref{claim1}. Pick $\delta=\delta(B,\mathcal L(V))$ as in Lemma \ref{l_polytope}. By passing to a subsequence, we may choose a $\mathbb Q$-divisor $0\leq G\in V$ such that $A^\circ$ is ample, $\rfdown B^\circ.=0$ and all $\rfdown B^\circ_m.=0$, where $A^\circ=A+G$, $B^\circ=B-G$ and $B^\circ_m=B_m-G$. By Steps 3 and 4, for infinitely many $m$ there exist $F_m\in\mathcal P_{A^\circ}(V)$ such that $B_m^\circ\in(B^\circ,F_m)$. In particular, setting $B_m'=F_m+G$, we have $B_m\in(B,B_m')$. Since $B-B_m'=B^\circ-F_m$, we may assume that $\|B-B_m'\|\leq\delta$ for $m\gg0$ by choosing $F_m$ closer to $B^\circ$ if necessary. Therefore, by Lemma \ref{l_polytope} applied to the polytope $\mathcal L(V)$ and the point $B\in\mathcal L(V)$, we have $B_m'\in\mathcal L(V)$ for $m\gg0$, and thus $B_m'\in\mathcal P_A(V)$ since $K_X+A+B_m'=K_X+A^\circ+F_m$ is numerically equivalent to an effective divisor. This finishes the proof of Claim \ref{claim1}. \\[2mm]
{\em Step 6.}
Therefore $\mathcal P_A(V)$ is a polytope, and we finally show that it is a {\em rational} polytope. Let $B_1,\dots,B_q$ be the extreme points of $\mathcal P_A(V)$.
Then there exist $\mathbb R$-divisors $D_i\geq0$ such that $K_X+A+B_i\equiv D_i$ for all $i$.
Let $W\subseteq \Div_{\mathbb R}(X)$ be the vector
space spanned by $V$ and by the components of $K_X+A$ and $\sum_{i=1}^q D_i$.
Note that for every  $\tau = (t_1,\dots,t_q)\in \mathbb R_+^q$ such that  $\sum t_i=1$, we have $B_\tau=\sum t_iB_i\in\mathcal P_A(V)$ and
$K_X+A+B_\tau\equiv\sum t_iD_i\in W$.
Let $\phi\colon W\longrightarrow N^1(X)_{\mathbb R}$ be the linear map sending an $\mathbb R$-divisor to its
numerical class. Then $W_0=\phi^{-1}(0)$ is a rational subspace of $W$ and
$$\mathcal P_A(V)=\{ B\in \mathcal L(V) \mid B=-K_X-A + D +
R\text{, where } 0\le D\in W, R\in W_0 \}.$$
Therefore, $\mathcal P_A(V)$ is cut out from $\mathcal L(V)\subseteq W$ by finitely
many rational half-spaces, and thus is a rational polytope.
\end{proof}

\section{Finite generation}
\label{s_finite}

In this section, we prove that Theorem \ref{t_cox}$_{n-1}$ and Theorem \ref{t_non-vanishing}$_{n}$ imply Theorem \ref{t_cox}$_{n}$; as an immediate consequence, we obtain Theorem \ref{t_finite}.


\begin{lemma}\label{l_cox}
Let $(X,\sum_{i=1}^p S_i)$ be a log smooth projective pair, let
$\mathcal{C}\subseteq \sum_{i=1}^p\mathbb R_+ S_i\subseteq \Div_\mathbb{R}(X)$ be a rational polyhedral cone, and
let $\mathcal{C}=\bigcup_{j=1}^p \mathcal C_j$ be a rational polyhedral decomposition.
Denote $\mathcal S=\mathcal C\cap\Div(X)$
and $\mathcal S_j=\mathcal C_j\cap\Div(X)$ for all $j$. Assume that:
\begin{enumerate}
\item[(i)] there exists $M>0$
such that, if $\sum\alpha_i S_i\in \mathcal{C}_j$ for some $j$ and some
$\alpha_i\in\mathbb N$ with $\sum\alpha_i\ge M$, then
$\sum\alpha_i S_i-S_j\in \mathcal{C}$;
\item[(ii)] the ring
$\rest_{S_j} R(X,\mathcal{S}_j)$ is finitely generated for every $j=1,\dots,p$.
\end{enumerate}
Then the divisorial ring $R(X,\mathcal{S})$ is finitely generated.
\end{lemma}
\begin{proof}
For every $i=1,\dots, p$, let $\sigma_i\in H^0(X,\ring X. (S_i))$ be a
section such that $\ddiv\sigma_i=S_i$. Let $\mathfrak R\subseteq R(X;S_1,\dots,S_p)$ be the ring
spanned by $R(X,\mathcal S)$ and $\sigma_1,\dots,\sigma_p$, and note that $\mathfrak R$ is graded by $\sum_{i=1}^p\mathbb N S_i$.
By Lemma \ref{l_gordan}(i), it is enough to show that $\mathfrak R$ is
finitely generated.

For any $\alpha=(\alpha_1,\dots, \alpha_p)\in \mathbb{N}^p$,
denote $D_\alpha=\sum \alpha_i S_i$ and $\deg(\alpha)=\sum \alpha_i$, and for a section
$\sigma\in H^0(X,\ring X. (D_\alpha))$, set $\deg(\sigma)=\deg(\alpha)$.
By (ii), for each $j=1,\dots,p$ there exists a finite set $\mathcal H_j\subseteq
R(X,\mathcal S_j)$ such that 
\begin{equation}\label{eq:65}
\rest_{S_j}R(X,\mathcal S_j)\quad\text{is generated by the set}\quad\{\sigma_{|S_j}\mid\sigma\in\mathcal H_j\}.
\end{equation}
Since the vector space $H^0(X,\ring X. (D_\alpha))$ is finite-dimensional for
every $\alpha\in\mathbb N^p$, there is a finite
set $\mathcal H\subseteq\mathfrak R$ such that
\begin{equation}\label{eq:64}
\{\sigma_1,\dots,\sigma_p\}\cup\mathcal H_1\cup\dots\cup\mathcal H_p\subseteq\mathcal H,
\end{equation}
and 
\begin{equation}\label{eq:63}
H^0(X,\mathcal O_X(D_\alpha))\subseteq\mathbb C[\mathcal H]\quad\text{for every }\alpha\in\mathbb N^p\,
\text{ with }D_\alpha\in\mathcal S\text{ and }\deg(\alpha)\leq M,
\end{equation}
where $\C[\mathcal H]$ is the $\mathbb C$-algebra generated by the elements of $\mathcal H$.
Observe that $\C[\mathcal H]\subseteq\mathfrak R$, and it suffices to show that
$\mathfrak R\subseteq\C[\mathcal H]$.

Let $\chi \in \mathfrak R$. By definition of
$\mathfrak R$, we may write
$\chi=\sum_i\sigma_1^{\lambda_{1,i}}\dots\sigma_p^{\lambda_{p,i}}\chi_i$, where
$\chi_i\in H^0(X,\ring X. (D_{\alpha_i}))$ for some $D_{\alpha_i}\in \mathcal
S$ and $\lambda_{j,i}\in \mathbb N$. Thus, it is enough to show that
$\chi_i\in\C[\mathcal H]$, and after replacing $\chi$ by $\chi_i$ we may assume
that 
$$\chi\in H^0(X,\mathcal O_X(D_\alpha))\quad\text{for some}\quad D_\alpha\in\mathcal S.$$
The proof is by induction on $\deg\chi$.
If $\deg\chi\leq M$, then $\chi\in\C[\mathcal H]$ by \eqref{eq:63}.
Now assume $\deg\chi>M$. Then there exists $1\leq j\leq p$ such that
$D_\alpha\in\mathcal S_j$, and so by \eqref{eq:65} and \eqref{eq:64} there are $\theta_1,\dots,\theta_z\in\mathcal H$ and a polynomial
$\varphi\in\C[X_1,\dots,X_z]$ such that $\chi_{|S_j}=\varphi(\theta_{1|S_j},\dots,\theta_{z|S_j})$. Therefore,
from the exact sequence
$$0\longrightarrow H^0(X,\ring X.(D_\alpha-S_j))\stackrel{\cdot \sigma_j}{\longrightarrow} H^0(X,\ring X.(D_\alpha))\longrightarrow
H^0(S_j,\ring {S_j}.(D_\alpha))$$
we obtain
$$\chi-\varphi(\theta_1,\dots,\theta_z)=\sigma_j\cdot\chi'\quad\text{for some}\quad \chi'\in H^0(X,\ring X.(D_\alpha-S_j)).$$
Note that $D_\alpha-S_j\in\mathcal S$ by (i), and
since $\deg\chi'<\deg\chi$, by induction we have $\chi'\in\C[\mathcal H]$. Therefore
$\chi=\sigma_j\cdot\chi'+\varphi(\theta_1,\dots,\theta_z)\in\C[\mathcal H]$, and we are done.
\end{proof}

\begin{lemma}\label{l_restricted}
Assume Theorem \ref{t_cox}$_{n-1}$ and Theorem \ref{t_non-vanishing}$_{n-1}$.

Let $(X,S+\sum_{i=1}^p S_i)$ be a log smooth projective pair of dimension $n$, where
$S$ and all $S_i$ are distinct prime divisors.
Let $V=\sum_{i=1}^p\mathbb RS_i\subseteq \Div_\mathbb{R}(X)$, let $A$ be an ample
$\mathbb{Q}$-divisor on $X$, let $B_1,\dots,B_m\in \mathcal
E_{S+A}(V)$ be $\mathbb Q$-divisors, and denote $D_i=K_X+S+A+B_i$.

Then the ring $\rest_S R(X;D_1,\dots,D_m)$ is finitely generated.
\end{lemma}
\begin{proof}
We first prove the lemma under an additional assumption that all $B_i$ lie in the interior of $\mathcal L(V)$ and that all $(S,B_{i|S})$ are terminal, and then treat the general case at the end of the proof.

Let $\mathcal G\subseteq\mathcal E_{S+A}(V)$ be the convex hull of all $B_i$. Then $\mathcal G$ is contained in the interior of $\mathcal L(V)$, and $(S,G_{|S})$ is terminal for every $G\in\mathcal G$.
Denote 
$$\mathcal F=\mathbb R_+(K_X+S+A+\mathcal G).$$
Then, by Lemma \ref{c_projection} it suffices to prove that $\rest_S R(X,\mathcal F)$ is finitely generated.

Let $W\subseteq \Div_{\mathbb R}(S)$  be the subspace spanned by the components of all $S_{i|S}$, and let  $\Phi_m$ and $\mathbf\Phi$ be the functions defined in Setup \ref{assumption}. By Theorem \ref{l_rational-polytope},
$\mathcal P=\mathcal G\cap\mathcal B_A^S(V)$ is a rational polytope, and $\mathbf\Phi$ extends to a rational piecewise affine function
on $\mathcal P$.
Thus, there exists a finite decomposition $\mathcal P=\bigcup\mathcal P_i$ into rational polytopes such that $\mathbf\Phi$
is rational affine on each $\mathcal P_i$. Denote $\mathcal C=\mathbb R_+(K_X+S+A+\mathcal P)$
and $\mathcal C_i=\mathbb R_+(K_X+S+A+\mathcal P_i)$, and note that $\mathcal C=\bigcup\mathcal C_i$.
Since $\rest_S H^0(X,\mathcal O_X(D))=0$ for every $D\in\mathcal F\setminus\mathcal C$, and as $\mathcal C$ is a rational polyhedral cone, it suffices to show that
$\rest_S R(X,\mathcal C)$ is finitely generated, and therefore, to prove that $\rest_S R(X,\mathcal C_i)$ is finitely generated for each $i$.
Hence, after replacing $\mathcal G$ by $\mathcal P_i$, we can assume that $\mathbf\Phi$ is rational affine on $\mathcal G$.

By Lemma \ref{l_g}, there exist $G_i\in\mathcal G\cap\Div_{\mathbb Q}(X)$
and $g_i\in\mathbb Q_+$,  with $i=1,\dots,q$,  such that $F_i=g_i(K_X+S+A+G_i)$  are  generators of $\mathcal F\cap\Div(X)$.
By Theorem \ref{l_rational-polytope}, there exists a positive integer $\ell$ such that $\Phi_m(G)={\mathbf\Phi}(G)$ for every
$G\in\mathcal G\cap\Div_{\mathbb Q}(X)$ and every $m\in\mathbb N$ such that $\frac m\ell G\in\Div(X)$.
Pick a positive integer $k$ such that all $\frac {kg_i}\ell\in\mathbb N$ and $\frac {kg_i}\ell G_i\in\Div(X)$.
For each nonzero $\alpha=(\alpha_1,\dots,\alpha_q)\in\mathbb N^q$, denote
$$g_\alpha=\sum\alpha_ig_i,\quad G_\alpha=\frac1{g_\alpha}\sum\alpha_i g_iG_i,\quad
F_\alpha=\sum\alpha_i F_i=g_\alpha(K_X+S+A+G_\alpha),$$
and note that $\frac{kg_\alpha}\ell G_\alpha\in\Div(X)$ and
${\mathbf\Phi}(G_\alpha)=\frac1{g_\alpha}\sum\alpha_ig_i{\mathbf\Phi}(G_i)$. Then, by Corollary \ref{c_inclusion} we have
\begin{align*}
\rest_S H^0(X,\ring X.(mkF_\alpha))&=H^0\big(S,\ring S. (mkg_\alpha(K_S+A_{|S}+\Phi_{mkg_\alpha}(G_\alpha)))\big)\\
&=H^0\big(S,\ring S. (mkg_\alpha(K_S+A_{|S}+{\mathbf\Phi}(G_\alpha)))\big)
\end{align*}
for all $\alpha\in\mathbb N^q$ and $m\in\mathbb N$, and thus
$$\rest_S R(X;kF_1,\dots,kF_q)=R(S;kg_1 F'_1,\dots,kg_q F'_q),$$
where $F'_i=K_S+A_{|S}+{\mathbf\Phi}(G_i)$. Since the last ring is a Veronese subring of the adjoint ring $R(S;F'_1,\dots,F'_q)$, it is
finitely generated by Theorem \ref{t_cox}$_{n-1}$ and by Lemma \ref{l_gordan}(i).
Therefore $\rest_S R(X;F_1,\dots,F_q)$ is finitely generated by Lemma \ref{l_gordan}(ii), and since there is the natural projection of this ring onto
$\rest_S R(X,\mathcal F)$, this proves the lemma under the additional assumption that all $B_i$ lie in the interior of $\mathcal L(V)$ and that all $(S,B_{i|S})$ are terminal.

In the general case, for every $i$ pick a $\mathbb Q$-divisor $G_i\in V$ such that $A-G_i$ is ample and $B_i+G_i$ is in the interior of $\mathcal L(V)$.
Let $A'$ be an ample $\mathbb Q$-divisor such that every $A-G_i-A'$ is also ample, and pick $\mathbb Q$-divisors $A_i\ge 0$ such that
$A_i\sim_\mathbb Q A-G_i-A'$, $\rfdown A_i.=0$,
$(X,S+\sum_{i=1}^p S_i+ \sum_{i=1}^m A_i)$ is log smooth, and the support of $\sum_{i=1}^m A_i$ does not contain
any of the divisors $S,S_1,\dots,S_p$. Let $V'\subseteq\Div_\mathbb R(X)$ be the vector space spanned by $V$ and by the components of $\sum_{i=1}^m A_i$.
Let $\varepsilon>0$ be a rational number such that $A''=A'-\varepsilon \sum_{i=1}^m A_i$ is ample, and such that $B_i'=B_i+G_i+A_i+\varepsilon\sum_{i=1}^m A_i$ is in the interior of $\mathcal L(V')$ for every $i$. 

Let $B\ge 0$ be a $\mathbb Q$-divisor such that $\rfdown B.=0$ and $B\ge B_i'$ for all $i$.
By Lemma \ref{l_disjoint}, there exists a log resolution $f\colon \map Y.X.$
such that
$$K_Y+T+C=f^*(K_X+S+B)+E,$$
where the $\mathbb Q$-divisors $C,E\ge 0$  have no common components, $E$ is $f$-exceptional, $\rfdown C.=0$, the
components of $C$ are disjoint, and $T=f^{-1}_*S\nsubseteq\Supp C$.
Then there are $\mathbb Q$-divisors $0\le C_i \le C$ and $f$-exceptional $\mathbb Q$-divisors $E_i\geq0$ such that
$$K_Y+T+C_i=f^*(K_X+S+B_i')+E_i,$$
and in particular, all pairs $(T,C_{i|T})$ are terminal.
Let $V^\circ$ be the subspace of $\Div_{\mathbb R}(Y)$ spanned by the components of $C$ and by all $f$-exceptional prime divisors.
There exists a $\mathbb Q$-divisor $F\ge 0$ on $Y$ such that $A^\circ$ is ample, every $C_i^\circ$ is in the interior of $\mathcal L(V^\circ)$, and every pair $(T,C^\circ_{i|T})$ is terminal, where $A^\circ=f^*A''-F$ and $C_i^\circ=C_i+F$. Denoting $D_i^\circ=K_Y+T+A^\circ+C_i^\circ$, it follows that
$$D_i^\circ\sim_\mathbb Q f^*D_i+E_i.$$
Then the ring $\rest_T R(Y;D_1^\circ,\dots,D_m^\circ)$ is finitely generated by the special case that we proved above, so $\rest_S R(X;D_1,\dots,D_m)$ is finitely generated by Corollary \ref{c_qlinear}.
\end{proof}

\begin{theorem} \label{t_b-to-a}
Theorem \ref{t_cox}$_{n-1}$ and Theorem \ref{t_non-vanishing}$_{n}$
imply Theorem \ref{t_cox}$_{n}$.
\end{theorem}
\begin{proof}
We first assume  that there exist $\mathbb Q$-divisors $F_i\ge 0$ such that 
\begin{equation}\label{eq:66}
\textstyle\big(X,\sum_i (B_i+F_i)\big)\text{ is log smooth and } K_X+A+B_i\sim_{\mathbb Q}F_i\text{ for every $i$}.
\end{equation}
We reduce the general case to this one at the end of the proof.

Let $W$ be the subspace of $\Div_{\mathbb{R}}(X)$ spanned by the components of all $B_i$ and $F_i$, and let $S_1,\dots,S_p$ be the prime divisors in $W$. Denote by $\mathcal T=\{(t_1,\dots,t_k)\mid t_i\geq0,\sum t_i=1\}\subseteq\mathbb R^k$
the standard simplex, and for each $\tau=(t_1,\dots,t_k)\in\mathcal T$, set
\begin{equation}\label{eq:69}
B_\tau=\sum_{i=1}^k t_iB_i \qquad\text{and}\qquad
F_\tau=\sum_{i=1}^k t_i F_i\sim_{\mathbb R}K_X+A+B_\tau.
\end{equation}
Denote
$$\mathcal B=\{ F_\tau+B\mid \tau \in \mathcal T,  0\leq B\in W, B_\tau + B \in \mathcal L(W)\}\subseteq W,$$
and for every $j=1,\dots,p$, let
$$\mathcal B_j=\{F_\tau + B\mid \tau \in \mathcal T,  0\leq B\in W, B_\tau + B \in \mathcal L(W), S_j\subseteq \rfdown B_\tau+B.\}\subseteq W.$$
Then $\mathcal B$ and $\mathcal B_j$ are rational polytopes, and thus $\mathcal C=\mathbb R_+\mathcal B$ and $\mathcal C_j=\mathbb R_+\mathcal B_j$ are
rational polyhedral cones. Denote $\mathcal S=\mathcal C\cap\Div(X)$ and $\mathcal S_j=\mathcal C_j\cap\Div(X)$. We claim that:
\begin{enumerate}
\item[(i)] $\mathcal C=\bigcup_{j=1}^p\mathcal C_j$,
\item[(ii)] there exists $M>0$
such that, if $\sum\alpha_i S_i\in \mathcal{C}_j$ for some $j$ and some
$\alpha_i\in\mathbb N$ with $\sum\alpha_i\ge M$, then
$\sum\alpha_i S_i-S_j\in \mathcal{C}$;
\item[(iii)] the ring
$\rest_{S_j} R(X,\mathcal{S}_j)$ is finitely generated for every $j=1,\dots,p$.
\end{enumerate}
This claim readily implies the theorem: indeed,
Lemma \ref{l_cox} then shows that $R(X,\mathcal S)$ is finitely generated. Let $d$ be a positive integer such that $F_i'=dF_i$ are integral divisors for $i=1,\dots,k$. Pick divisors $F_{k+1}',\dots,F_m'$
such that $F_1',\dots,F_m'$ are generators of $\mathcal S$. Then $R(X;F_1',\dots,F_m')$ is finitely
generated by Lemma \ref{c_projection}, and so is $R(X;F_1',\dots,F_k')$ by Lemma \ref{l_gordan}(i). Finally, Lemma \ref{l_gordan}(ii) implies that $R(X;F_1,\dots,F_k)$ is finitely generated, and therefore so is $R(X;D_1,\dots,D_k)$ by \eqref{eq:66} and by Corollary \ref{c_qlinear}.

We now prove the claim. In order to see (i), fix $G\in\mathcal C\backslash\{0\}$. Then, by definition of $\mathcal C$,
there exist $\tau\in\mathcal T, 0\leq B\in W$ and $r>0$ such that
$B_\tau+B\in \mathcal L(W)$ and $G=r(F_\tau+B)$.
Setting
$$\lambda=\max \{t \ge 1\mid B_\tau+ t B+ (t-1)F_\tau\in \mathcal L(W) \}
$$
and $B'=\lambda B +(\lambda -1 )F_\tau$, we have
$$\lambda G = r(F_\tau + B'),$$
and there exists $j_0$ such that $S_{j_0}\subseteq \rfdown B_\tau+B'.$. Therefore
$G\in\mathcal C_{j_0}$, which proves (i).

For (ii), note first that there exists $\varepsilon>0$ such that $\|B_i\|\le 1-\varepsilon$ for all $i$, and thus
\begin{equation}\label{eq:68}
\|B_\tau\|\le1-\varepsilon\quad\text{for any }\tau\in \mathcal T.
\end{equation}
Since the polytopes $\mathcal B_j\subseteq W$ are compact, there is a positive constant $C$ such that $\|\Psi\|\leq C$
for any $\Psi\in\bigcup_{j=1}^p\mathcal B_j$, and denote $M=pC/\varepsilon$.
For some $j\in\{1,\dots,p\}$, let  $G=\sum\alpha_iS_i\in \mathcal S_j$ be such that  $\sum\alpha_i\geq M$. Since $p\|G\|\geq\sum\alpha_i$, we have
$$\|G\|\ge \frac M p=\frac C \varepsilon.$$
By definition of $\mathcal C_j$ and of $C$, we may write $G=rG'$ with $G'\in \mathcal B_j$, $\|G'\|\leq C$ and $r>0$. In particular, 
\begin{equation}\label{eq:67}
r=\frac{\|G\|}{\|G'\|}\ge \frac 1\varepsilon.
\end{equation}
Furthermore, $G'=F_\tau+B$ for some $\tau\in \mathcal T$ and $0\leq B\in W$ such that $B_\tau+B\in\mathcal L(W)$ and
$S_j\subseteq \rfdown B_\tau+B.$. Therefore, by \eqref{eq:68} and \eqref{eq:67} we have
$$\mult_{S_j}B=1-\mult_{S_j}B_\tau\ge \varepsilon\ge \frac 1 r,$$
and thus
$$G-S_j = r\big(F_\tau+ B - \textstyle\frac 1 r S_j\big)\in \mathcal C.$$

Finally, to show (iii), fix $j\in\{1,\dots,p\}$, and let $\{E_1,\dots,E_\ell\}$ be a set of generators of
$\mathcal S_j$. Then, by definition of $\mathcal S_j$ and by \eqref{eq:69}, for every $i=1,\dots,\ell$, there exist $k_i\in\mathbb Q_+$, $\tau_i\in\mathcal T\cap\mathbb Q^k$ and $0\leq B_i\in W$ such that
$B_{\tau_i}+B_i\in\mathcal L(W)$, $S_j\subseteq \rfdown B_{\tau_i}+B_i.$ and
$$E_i=k_i(F_{\tau_i}+B_i)\sim_\mathbb Q k_i(K_X+A+B_{\tau_i}+B_i).$$
Denote $E_i'=K_X+A+B_{\tau_i}+B_i$. Then the ring $\rest_{S_j}R(X;E'_1,\dots,E'_\ell)$ is finitely generated by Lemma \ref{l_restricted},
and thus so is $\rest_{S_j}R(X;E_1,\dots,E_\ell)$ by Corollary \ref{c_qlinear}.
Since there is the natural projection $\rest_{S_j}R(X;E_1,\dots,E_\ell)\longrightarrow \rest_{S_j} R(X,\mathcal S_j)$,
this completes the proof under the additional assumption that \eqref{eq:66} holds.

We now prove the general case. Let $V$ be the subspace of $\Div_{\mathbb{R}}(X)$ spanned by the components of all $B_i$, let $\mathcal P\subseteq V$ be the convex hull
of all $B_i$, and denote $\mathcal R=\mathbb R_+(K_X+A+\mathcal P)$. Then, by Lemma \ref{c_projection} it suffices to show that $R(X,\mathcal R)$ is finitely generated. By Theorem \ref{t_non-vanishing}$_n$, $\mathcal P_\mathcal E=\mathcal P\cap\mathcal{E}_A(V)$ is a rational polytope, and denote $\mathcal R_\mathcal E=\mathbb R_+(K_X+A+\mathcal P_\mathcal E)$. Since $H^0(X,\mathcal O_X(D))=0$ for every integral divisor $D\in\mathcal R\setminus\mathcal R_\mathcal E$, the ring $R(X,\mathcal R)$ is finitely generated if and only if $R(X,\mathcal R_\mathcal E)$ is. 

By Lemma \ref{l_g}, the monoid $\mathcal R_\mathcal E\cap\Div(X)$ is finitely generated, and let $R_i$ be its generators for $i=1,\dots,\ell$.
Then there exist $p_i\in\mathbb Q_+$ and $P_i\in\mathcal P_\mathcal E\cap\Div_\mathbb Q(X)$ such that $R_i=p_i(K_X+A+P_i)$. By construction, $\rfdown P_i.=0$ and there exist $\mathbb Q$-divisors $G_i\ge 0$ such that 
$$K_X+A+P_i\sim_{\mathbb Q}G_i$$
for all $i$. Let $f\colon Y\longrightarrow X$ be a log resolution of $\big(X,\sum_i (P_i+G_i)\big)$. For every $i$,
there are $\mathbb Q$-divisors $C_i,E_i\geq0$ on $Y$ with no common components such that $E_i$ is $f$-exceptional and
$$
K_Y+C_i=f^*(K_X+P_i)+E_i.
$$
Note that $\rfdown C_i.=0$, and denote $F_i^\circ=p_i(f^*G_i+E_i)\geq0$. Let $H\geq0$ be an $f$-exceptional $\mathbb Q$-divisor on $Y$ such that $A^\circ$ is ample and $\rfdown C_i^\circ.=0$ for all $i$, where $A^\circ=f^*A-H$ is ample and $C_i^\circ=C_i+H$, and denote $D_i^\circ=K_Y+A^\circ+C_i^\circ$. Then
$$
p_iD_i^\circ\sim_{\mathbb Q}f^*R_i+p_iE_i\sim_\mathbb Q F_i^\circ.
$$
This last relation implies two things: first, it follows from what we proved above and by Lemma \ref{l_gordan} that the adjoint ring $R(Y;D_1^\circ,\dots,D_\ell^\circ)$ is finitely generated. Second, the ring $R(X;R_1,\dots,R_\ell)$ is then finitely generated by Corollary \ref{c_qlinear}. Since there is the natural projection map $R(X;R_1,\dots,R_\ell)\longrightarrow R(X,\mathcal R_\mathcal E)$, the ring $R(X,\mathcal R_\mathcal E)$ is finitely generated, and we are done.
\end{proof}

Finally, we have:

\begin{proof}[Proof of Theorem \ref{t_finite}]
By Theorem \ref{t_fujinomori}, there exist a projective klt pair $(Y,\Gamma)$
and positive integers $p$ and $q$ such that $p(K_X+\Delta)$ and $q(K_Y+\Gamma)$ are integral, $K_Y+\Gamma$ is big and
$R(X,p(K_X+\Delta))\simeq R(Y,q(K_Y+\Gamma))$.
Write $K_Y+\Gamma\sim_\mathbb{Q} A+B$, where $A$ is an ample $\mathbb Q$-divisor and $B\geq0$.
Let $f\colon Y'\longrightarrow Y$ be a log resolution of $(Y,\Gamma+B)$, let $\Gamma',E\geq0$ be $\mathbb Q$-divisors such that $E$ is
$f$-exceptional and $K_{Y'}+\Gamma'=f^*(K_Y+\Gamma)+E$, and let $H\geq0$ be an $f$-exceptional $\mathbb Q$-divisor such that $A'=f^*A-H$ is ample.
Pick a rational number $0<\varepsilon\ll1$ such that if
$C=\Gamma'+\varepsilon f^*B+\varepsilon H$, then $\rfdown C.=0$, and note that
$K_{Y'}+C+\varepsilon A'\sim_\mathbb{Q}(\varepsilon+1)f^*(K_Y+\Gamma)+E$.
Then the ring $R(Y,K_Y+\Gamma)$ is finitely generated by Theorem \ref{t_cox} and Corollary \ref{c_qlinear}, and thus so is
$R(X,K_X+\Delta)$ by Lemma \ref{l_gordan}.
\end{proof}

\bibliographystyle{amsalpha}
\bibliography{Library}

\newcommand{\etalchar}[1]{$^{#1}$}
\providecommand{\bysame}{\leavevmode\hbox to3em{\hrulefill}\thinspace}
\providecommand{\MR}{\relax\ifhmode\unskip\space\fi MR }
\providecommand{\MRhref}[2]{%
  \href{http://www.ams.org/mathscinet-getitem?mr=#1}{#2}
}
\providecommand{\href}[2]{#2}
\begin{thebibliography}{BCHM10}

\bibitem[ADHL10]{ADHL10}
I.~Arzhantsev, U.~Derenthal, J.~Hausen, and A.~Laface, \emph{Cox {R}ings},
  arXiv:1003.4229v1\setbox0=\hbox{2010}.

\bibitem[BCHM10]{BCHM10}
C.~Birkar, P.~Cascini, C.~Hacon, and J.~M\textsuperscript{c}Kernan,
  \emph{Existence of minimal models for varieties of log general type}, J.
  Amer. Math. Soc. \textbf{23} (2010), no.~2, 405--468.

\bibitem[Bou89]{Bourbaki89}
N.~Bourbaki, \emph{Commutative algebra. {C}hapters 1--7}, Elements of
  Mathematics, Springer-Verlag, Berlin, 1989.

\bibitem[CL10]{CL10}
A.~Corti and V.~Lazi\'c, \emph{New outlook on the {M}inimal {M}odel {P}rogram,
  {II}}, arXiv:1005.0614v2\setbox0=\hbox{2010}.

\bibitem[CL11]{CL11}
P.~Cascini and V.~Lazi\'c, \emph{The {M}inimal {M}odel {P}rogram {R}evisited},
  to appear in Contributions to Algebraic Geometry (P.\ Pragacz, ed.), EMS
  Publishing House\setbox0=\hbox{2011}.

\bibitem[Cor07]{Corti05}
A.~Corti, \emph{3-fold flips after {S}hokurov}, Flips for 3-folds and 4-folds
  (A.~Corti, ed.), Oxford University Press, 2007, pp.~13--40.

\bibitem[Cor11]{Corti11}
\bysame, \emph{Finite generation of adjoint rings after {L}azi\'c: an
  introduction}, Classification of Algebraic Varieties (C.~Faber, G.~van~der
  Geer, and E.~Looijenga, eds.), EMS Series of Congress Reports, EMS Publishing
  House, 2011, pp.~197--220.

\bibitem[ELM{\etalchar{+}}06]{ELMNP}
L.~Ein, R.~Lazarsfeld, M.~Musta{\c{t}}\u{a}, M.~Nakamaye, and M.~Popa,
  \emph{Asymptotic invariants of base loci}, Ann. Inst. Fourier (Grenoble)
  \textbf{56} (2006), no.~6, 1701--1734.

\bibitem[FM00]{FM00}
O.~Fujino and S.~Mori, \emph{A canonical bundle formula}, J. Differential Geom.
  \textbf{56} (2000), no.~1, 167--188.

\bibitem[Ful93]{Fulton93}
W.~Fulton, \emph{Introduction to toric varieties}, Princeton University Press,
  1993.

\bibitem[HK10]{HK10}
C.~Hacon and S.~Kov\'acs, \emph{Classification of {H}igher {D}imensional
  {A}lgebraic {V}arieties}, Oberwolfach Seminars, vol.~41, Birkh\"auser, Basel,
  2010.

\bibitem[HM10]{HM10}
C.~Hacon and J.~M\textsuperscript{c}Kernan, \emph{Existence of minimal models
  for varieties of log general type {II}}, J. Amer. Math. Soc. \textbf{23}
  (2010), no.~2, 469--490.

\bibitem[KM98]{KM98}
J.~Koll{\'a}r and S.~Mori, \emph{Birational {G}eometry of {A}lgebraic
  {V}arieties}, Cambridge {T}racts in {M}athematics, vol. 134, Cambridge
  University Press, 1998.

\bibitem[Laz09]{Laz09}
V.~Lazi{\'c}, \emph{Adjoint rings are finitely generated},
  arXiv:0905.2707v3\setbox0=\hbox{2009}.

\bibitem[Mor82]{Mori82}
S.~Mori, \emph{Threefolds whose canonical bundles are not numerically
  effective}, Ann. of Math. (2) \textbf{116} (1982), no.~1, 133--176.

\bibitem[Nak04]{Nakayama04}
N.~Nakayama, \emph{Zariski-decomposition and abundance}, MSJ Memoirs, vol.~14,
  Mathematical Society of Japan, Tokyo, 2004.

\bibitem[P{\u{a}}u08]{Pau08}
M.~P{\u{a}}un, \emph{Relative critical exponents, non-vanishing and metrics
  with minimal singularities}, arXiv:0807.3109v1\setbox0=\hbox{2008}.

\bibitem[Sho03]{Shokurov03}
V.~V. Shokurov, \emph{Prelimiting flips}, Proc. Steklov Inst. of Math.
  \textbf{240} (2003), 82--219.

\bibitem[Siu98]{Siu98}
Y.-T. Siu, \emph{Invariance of plurigenera}, Invent. Math. \textbf{134} (1998),
  no.~3, 661--673.

\bibitem[Siu08]{Siu08}
\bysame, \emph{Finite generation of canonical ring by analytic method}, Sci.
  China Ser. A \textbf{51} (2008), no.~4, 481--502.

\end{thebibliography}

\end{document}